\documentclass[12pt]{amsart}
\usepackage{amsmath}
\usepackage{amstext}
\usepackage{amssymb}
\usepackage{amsthm}
\swapnumbers
\theoremstyle{plain}%default
\newtheorem{thm}{Theorem}[section]
\newtheorem{lem}[thm]{Lemma}

\newtheorem{prop}[thm]{Proposition}
\newtheorem{cor}[thm]{Corollary}

\theoremstyle{definition}
\newtheorem{defi}[thm]{Definition}
\newtheorem{defs}[thm]{Definitions}

\newtheorem{ex}[thm]{Example}
\newtheorem{exs}[thm]{Examples}

\newtheorem{ntn}[thm]{Notation}

\newtheorem{rmds}[thm]{Reminders}
\theoremstyle{remark}
\newtheorem*{note}{Note}
\newtheorem{rmk}[thm]{Remark}
\newtheorem{disc}[thm]{Discussion}

\DeclareMathOperator{\height}{ht} 
 
\DeclareMathOperator{\ext}{exten} 
 
\DeclareMathOperator{\Ker}{Ker}

\DeclareMathOperator{\Supp}{Supp} 
 
   \DeclareMathOperator{\Max}{Max}

 \DeclareMathOperator{\Var}{Var}
\DeclareMathOperator{\HSL}{HSL} \DeclareMathOperator{\Spec}{Spec}
\DeclareMathOperator{\Quot}{Quot} 
 \DeclareMathOperator{\ann}{ann}
\DeclareMathOperator{\grann}{gr-ann} \DeclareMathOperator{\Min}{Min}
\def\Z{\mathbb Z}
\def\N{\mathbb N}

\def\F{\mathbb F}
\def\K{\mathbb K}

\def\fa{{\mathfrak{a}}}
\def\fA{{\mathfrak{A}}}
\def\fb{{\mathfrak{b}}}
\def\fB{{\mathfrak{B}}}
\def\fc{{\mathfrak{c}}}
\def\fC{{\mathfrak{C}}}
\def\fd{{\mathfrak{d}}}
\def\fh{{\mathfrak{h}}}
\def\fk{{\mathfrak{k}}}
\def\m{{\mathfrak{m}}}
\def\fm{{\mathfrak{m}}}
\def\fM{{\mathfrak{M}}}

\def\fn{{\mathfrak{n}}}

\def\fp{{\mathfrak{p}}}
\def\fP{{\mathfrak{P}}}
\def\fq{{\mathfrak{q}}}

\def\ft{{\mathfrak{t}}}

\def\nn{\relax\ifmmode{\mathbb N_{0}}\else$\mathbb N_{0}$\fi}
\def\lra{\longrightarrow}

\begin{document}

\title[Big test
elements for some excellent rings]{Big tight closure test elements for\\
some non-reduced excellent rings}
\author{RODNEY Y. SHARP}
\address{{\it School of Mathematics and Statistics\\
University of Sheffield\\ Hicks Building\\ Sheffield S3 7RH\\ United
Kingdom}} \email{R.Y.Sharp@sheffield.ac.uk}

\subjclass[2010]{Primary 13A35, 16S36, 13D45, 13E05, 13E10, 13H10;
Secondary 13J10}

\date{\today}

\keywords{Commutative Noetherian ring, prime characteristic,
Frobenius homomorphism, tight closure, test element, excellent ring,
condition $(R_0)$, Frobenius skew polynomial ring, Gorenstein ring,
weakly $F$-regular ring.}

\begin{abstract} This paper is concerned with existence of big tight closure test
elements for a commutative Noetherian ring $R$ of prime
characteristic $p$. Let $R^{\circ}$ denote the complement in $R$ of
the union of the minimal prime ideals of $R$. A big test element for
$R$ is an element of $R^{\circ}$ which can be used in every tight
closure membership test for {\em every\/} $R$-module, and not just
the finitely generated ones. The main results of the paper are that,
if $R$ is excellent and satisfies condition $(R_0)$, and $c \in
R^{\circ}$ is such that $R_c$ is Gorenstein and weakly $F$-regular,
then some power of $c$ is a big test element for $R$ if (i) $R$ is a
homomorphic image of an excellent regular ring of characteristic $p$
for which the Frobenius homomorphism is intersection-flat, or (ii)
$R$ is $F$-pure, or (iii) $R$ is local. The Gamma construction is
not used.
\end{abstract}

\maketitle

\setcounter{section}{-1}
\section{\it Introduction}
\label{intro}

This paper is concerned with existence of tight closure test
elements in certain (commutative Noetherian) rings of prime
characteristic.

Throughout the paper, $R$ will denote a commutative Noetherian ring
of prime characteristic $p$. We shall only assume that $R$ is local
when this is explicitly stated; then, the notation `$(R,\fm)$' will
denote that $\fm$ is the maximal ideal of $R$. We use $R^{\circ}$ to
denote the complement in $R$ of the union of the minimal prime
ideals of $R$. Let $\fa$ be a proper ideal of $R$. For $n \in \nn$
(we use $\nn$ (respectively $\N$) to denote the set of non-negative
(respectively positive) integers), the {\em $n$th Frobenius power\/}
$\fa^{[p^n]}$ of $\fa$ is the ideal of $R$ generated by all $p^n$th
powers of elements of $\fa$.

An element $r \in R$ belongs to the {\em tight closure $\fa^*$ of
$\fa$\/} if and only if there exists $c \in R^{\circ}$ such that
$cr^{p^n} \in \fa^{[p^n]}$ for all $n \gg 0$. We say that $\fa$ is
{\em tightly closed\/} if $\fa^* = \fa$. The ring $R$ is said to be
{\em weakly $F$-regular\/} if every ideal of $R$ is tightly closed.
Regular rings of characteristic $p$ are weakly $F$-regular. The
theory of tight closure was invented by M. Hochster and C. Huneke
\cite{HocHun90}, and many applications have been found for the
theory: see \cite{Hunek96} and \cite{Hunek98}.

The concept of test element plays an important r\^ole in the theory.
A {\em test element for ideals\/} for $R$ is an element $c' \in
R^{\circ}$ such that, for {\em every\/} ideal $\fb$ of $R$ and {\em
every\/} $r \in \fb^*$ it is the case that $c'r^{p^n} \in
\fb^{[p^n]}$ for {\em all\/} $n \in \nn$.

We are now going to reformulate the definitions of tight closure and
test element for ideals in terms of the Frobenius skew polynomial
ring. We shall always denote by $f:R\lra R$ the Frobenius
homomorphism, for which $f(r) = r^p$ for all $r \in R$. The {\em
Frobenius skew polynomial ring over $R$\/} is the skew polynomial
ring $R[x,f]$ associated to $R$ and $f$ in the indeterminate $x$
over $R$. Recall that $R[x,f]$ is, as a left $R$-module, freely
generated by $(x^i)_{i \in \nn}$,
 and so consists
 of all polynomials $\sum_{i = 0}^n r_i x^i$, where  $n \in \nn$
 and  $r_0,\ldots,r_n \in R$; however, its multiplication is subject to the
 rule
 $
  xr = f(r)x = r^px$ for all $r \in R.$
Note that $R[x,f]$ can be considered as a positively-graded ring
$R[x,f] = \bigoplus_{n=0}^{\infty} R[x,f]_n$, with $R[x,f]_n = Rx^n$
for $n \in \nn$. If we endow $Rx^n$ with its natural structure as an
$(R,R)$-bimodule (inherited from its being a graded component of
$R[x,f]$), then $Rx^n$ is isomorphic (as $(R,R)$-bimodule) to $R$
viewed as a left $R$-module in the natural way and as a right
$R$-module via $f^n$, the $n$th iterate of the Frobenius ring
homomorphism.

Now $R[x,f]\otimes_R(R/\fa)$ can be viewed as an $\nn$-graded left
$R[x,f]$-module
$$
R[x,f]\otimes_R(R/\fa) = \bigoplus_{n\in\nn}Rx^n\otimes_R(R/\fa)
$$ in a natural way: for $r,r' \in R$ and $j \in \nn$, we have
$x(rx^j\otimes (r' + \fa)) = xrx^j\otimes (r' + \fa) =
r^px^{j+1}\otimes (r' + \fa).$ Since the $R$-module
$Rx^n\otimes_R(R/\fa)$ (for $n \in \nn$) is isomorphic to
$R/\fa^{[p^n]}$, we can reformulate the definitions of tight closure
and test element for ideals in the following ways. An element $r \in
R$ belongs to the tight closure $\fa^*$ of $\fa$ if and only if
there exists $c \in R^{\circ}$ such that, in the left
$R[x,f]$-module $R[x,f]\otimes_R(R/\fa)$, the element $1 \otimes (r
+ \fa) \in (R[x,f]\otimes_R(R/\fa))_0$ is annihilated by $cx^n$ for
all $n \gg 0$, while a test element for ideals for $R$ is an element
$c' \in R^{\circ}$ such that, for {\em every\/} ideal $\fb$ of $R$
and {\em every\/} $r \in \fb^*$, the element $1 \otimes (r + \fb)
\in (R[x,f]\otimes_R(R/\fb))_0$ is annihilated by $c'x^n$ for {\em
all\/} $n \in\nn$.

Let $L$ and $M$ be $R$-modules and let $K$ be a submodule of $L$.
Again there is a natural structure as $\nn$-graded left
$R[x,f]$-module on $R[x,f]\otimes_RM =
\bigoplus_{n\in\nn}(Rx^n\otimes_RM)$. An element $m \in M$ belongs
to $0^*_M$, the {\em tight closure of the zero submodule in $M$\/},
if and only if there exists $c \in R^{\circ}$ such that the element
$1 \otimes m \in (R[x,f]\otimes_RM)_0$ is annihilated by $cx^j$ for
all $j \gg 0$: see Hochster--Huneke \cite[\S 8]{HocHun90}.
(Incidentally, in discussions of this type, we shall often tacitly
identify
$$(R[x,f]\otimes_RM)_0 = Rx^0\otimes_RM = R\otimes_RM$$ with $M$ in
the obvious way.) Furthermore, the tight closure $K^*_L$ of $K$ in
$L$ is the inverse image, under the natural epimorphism $L \lra
L/K$, of $0^*_{L/K}$, the tight closure of $0$ in $L/K$. In general,
we have $K \subseteq K^*_L$; we say that $K$ is {\em tightly closed
in $L$\/} if $K = K^*_L$, that is, if and only if $0^*_{L/K} = 0$.

A {\em test element for modules\/} for $R$ is an element $c' \in
R^{\circ}$ such that, for {\em every finitely generated\/}
$R$-module $M$ and {\em every\/} $j \in \nn$, the element $c'x^j$
annihilates $1 \otimes m \in (R[x,f]\otimes_RM)_0$ for {\em every\/}
$m \in 0^*_M$. Note that a test element for ideals for $R$ is an
element $c' \in R^{\circ}$ such that, for {\em every cyclic\/}
$R$-module $M$ and {\em every\/} $j \in \nn$, the element $c'x^j$
annihilates $1 \otimes m \in (R[x,f]\otimes_RM)_0$ for {\em every\/}
$m \in 0^*_M$. Recall that a commutative ring is said to be {\em
reduced\/} if it has no non-zero nilpotent element. When $R$ is
reduced and excellent, the concepts of test element for modules and
test element for ideals for $R$ coincide: see \cite[Discussion (8.6)
and Proposition (8.15)]{HocHun90}.

A {\em big test element\/} for $R$ is defined to be an element $c
\in R^{\circ}$ such that, for {\em every\/} $R$-module $M$ and {\em
every\/} $j \in \nn$, the element $cx^j$ annihilates $1 \otimes m
\in (R[x,f]\otimes_RM)_0$ for {\em every\/} $m \in 0^*_M$. A big
test element $c$ for $R$ is said to be a {\em locally stable
(respectively, completely stable) big test element for $R$} if and
only if, for every $\fp \in \Spec (R)$, the natural image $c/1$ of
$c$ in $R_{\fp}$ (respectively, $\widehat{R_{\fp}}$) is a big test
element for $R_{\fp}$ (respectively, for $\widehat{R_{\fp}}$).

We note here that the definition of $F$-purity can be formulated in
terms of left modules over the Frobenius skew polynomial ring: $R$
is {\em $F$-pure\/} if, for each $R$-module $N$, the map $\psi_N : N
\lra Rx\otimes_RN$ for which $\psi_N(g) = x\otimes g$ for all $g \in
N$ is injective.

It has been generally accepted that the best known results about the
existence of test elements are those of Hochster--Huneke in \cite[\S
6]{HocHun94}.  Here is one of their results. (For $c\in R$, we
denote by $R_c$ the ring of fractions of $R$ with respect to the
multiplicatively closed set consisting of the powers of $c$.)

\begin{thm}[M. Hochster and C. Huneke {\cite[Theorem
(6.1)(b)]{HocHun94}}] \label{in.1} {\rm (See also \cite[p.\
18]{Hunek96}.)} A reduced algebra $R$ of finite type over an
excellent local ring of characteristic $p$ has a test element for
modules.

In fact, if $c \in R^{\circ}$ is such that $R_c$ is Gorenstein and
weakly $F$-regular (and, in particular, if $c \in R^{\circ}$ is such
that $R_c$ is regular), then some power of $c$ is a test element for
modules for $R$.
\end{thm}

However, it is also generally accepted that the proof of the above
theorem in \cite{HocHun94} is quite difficult, because it depends on
the so-called `Gamma construction' of \cite[\S 6]{HocHun94}.

In \cite[p.\ 201]{Hunek98}, in his discussion of test elements,
Huneke writes as follows. `Of course, the best theorem [about test
elements] one wants is one which says that if $R$ is reduced (and
excellent, perhaps) and if [$c \in R^{\circ}$ is such that] $R_c$ is
weakly $F$-regular, then $c$ has a power which is a test element.
This is an open question.'

This paper and paper \cite{Fpurhastest} arose from attempts to make
some progress on that open question, and also, at the same time, to
attempt to provide proofs about existence of test elements that do
not use the Gamma construction. The main result of
\cite{Fpurhastest} is the following.

\begin{thm}[R. Y. Sharp {\cite[Theorem
4.16]{Fpurhastest}}] \label{in.11} If $R$ is $F$-pure and excellent
(but not necessarily local), then $R$ has a big test element.

In fact, if $c \in R^{\circ}$ is such that $R_c$ is regular, then
$c$ itself is a big test element for $R$.
\end{thm}

The method of proof of Theorem \ref{in.11} in \cite{Fpurhastest} can
avoid use of the Gamma construction; the methods of
\cite{Fpurhastest} use the theory of graded annihilators of left
modules over the Frobenius skew polynomial ring developed in
\cite{ga}.

The main results of this paper establish the existence of big test
elements for certain excellent rings that satisfy the condition
$(R_0)$. Recall that $R$ is said to satisfy the condition $(R_0)$ if
and only if $R_{\fp}$ is a field for every minimal prime ideal $\fp$
of $R$, or, equivalently, if all the primary components
corresponding to minimal prime ideals of the zero ideal of $R$ are
prime. The main result of this paper in the local case is the
following.

\vspace{0.1in}

\noindent{\bf Theorem.} {\it Suppose that $(R,\fm)$ is excellent and
local, and satisfies the condition $(R_0)$. Then $R$ has a big test
element.

In fact, for each $c \in R^{\circ}$ for which $R_c$ is Gorenstein
and weakly $F$-regular, some power of $c$ is a big test element for
$R$.}

\vspace{0.1in}

This result is presented in Theorem \ref{jan.17gf} below. The
preparatory Theorem \ref{jan.17} establishes that, for an excellent
local $R$ satisfying $(R_0)$, if $c' \in R^{\circ}$ is such that
$R_{c'}$ is regular, then some power of $c'$ is a big test element
for $R$; the latter theorem is used in \S \ref{wys} to produce a
test element that enables one to show, without use of the Gamma
construction, that if $R_c$ is Gorenstein and weakly $F$-regular,
then $\widehat{R}_c$ is Gorenstein and weakly $F$-regular too; the
proof of Theorem \ref{jan.17gf} can then be completed.

Theorem \ref{jan.17gf} is not covered by the Hochster--Huneke
Theorem \ref{in.1} mentioned above because, firstly, it establishes
the existence of test elements in an excellent local ring (of
characteristic $p$) which satisfies the condition $(R_0)$, and such
a ring need not be reduced, and, secondly, it asserts the existence
of big test elements. However, with reference to this second point,
although \cite[Theorem (6.1)(b)]{HocHun94} does not mention big test
elements, there is a version of the Hochster--Huneke Theorem stated
on page 132 of a manuscript (from a lecture course) entitled
`Foundations of tight closure theory' on Hochster's homepage at
%\begin{center}
\begin{verbatim} http://www.math.lsa.umich.edu/~hochster/711F07/fndtc.pdf
\end{verbatim}
%\end{center}
and that version indicates that Hochster and Huneke are aware that
their methods yield big test elements.

Some comments about the hypothesis that $R$ satisfies the condition
$(R_0)$ (in Theorem \ref{jan.17gf}) are in order, because part of
the `folklore' that has grown up around tight closure is the
statement that `if $R$ has a test element, then it must be reduced':
see \cite[p.\ 20]{Hunek96}, \cite[p.\ 201]{Hunek98} and \cite[p.\
3152]{LyuSmi01}. In fact, this statement is not true, as the
following example shows.

\begin{ex}\label{in.2} Let $S$ be the ring of formal power series
$\F_2[[W,Y]]$ over the field $\F_2$ of two elements, and set $\fa :=
(W^2,WY)$ and $R := S/\fa$. Note that $R$ is not reduced. We show
that the natural image $y$ of $Y$ in $R$ is a test element for
ideals for $R$.

Let $w$ denote the natural image of $W$ in $R$. Each element of $R$
can be uniquely written in the form $\beta w +
\sum_{i=0}^{\infty}\alpha_iy^i$ where $\beta, \alpha_0, \alpha_1,
\alpha_2, \ldots \in \F_2$. It follows easily that the nilradical
$\fn$ of $R$ is $Rw$, and that $R/\fn = R/Rw \cong \F_2[[Z]]$ (where
$Z$ is a new variable), a regular local ring. Therefore each ideal
of $R/\fn$ is tightly closed. A consequence is that, for each ideal
$\fb$ of $R$, we have $\fb^* = \fb + Rw$. Thus, in order to show
that $y$ is a test element for ideals for $R$, it is sufficient to
show that, for each ideal $\fb$ of $R$ that does not contain $w$, we
have $yw^{2^n} \in \fb^{[2^n]}$ for all $n \in \nn$. However, this
is obvious, because $yw= 0$.
\end{ex}

The lemma below shows that if $R$ has a test element for ideals,
then it satisfies the condition $(R_0)$.  This condition is weaker
than `$R$ is reduced'.

\begin{lem}\label{in.3} If $R$ has a test element $c$ for ideals,
then $R$ satisfies the condition $(R_0)$.
\end{lem}

\begin{proof} Let $0 = \bigcap_{i=1}^{h+t} \fq_i$ be a minimal
primary decomposition for the zero ideal of $R$, with $\sqrt{\fq_i}
= \fp_i$ for all $i = 1, \ldots, h+t$, where the numbering is such
that $\fp_1, \ldots, \fp_h$ are the minimal prime ideals of $0$ and
$\fp_{h+1}, \ldots, \fp_{h+t}$ are the embedded prime ideals of $0$.
(Of course, $t$ could be zero.) Now the nilradical $\fn$ of $R$ is
contained in the tight closure of every ideal of $R$, including
$0^*$. Let $r \in \fn$. Since $c$ is a test element for ideals for
$R$, we have $cr^{p^n} \in 0^{[p^n]} = 0$ for all $n \in \nn$. In
particular, $cr \in 0 = \bigcap_{i=1}^{h+t} \fq_i$. But $c \in
R^{\circ} = R \setminus \bigcup_{i=1}^h\fp_i$. Therefore $r \in
\bigcap_{i=1}^{h} \fq_i$. It thus follows that $ \fn =
\bigcap_{i=1}^h \fp_i \subseteq \bigcap_{i=1}^{h} \fq_i. $ The
reverse inclusion is obvious. We then conclude from the uniqueness
theorems for primary decomposition that $\fq_i = \fp_i$ for all $i =
1, \ldots, h$.
\end{proof}

As mentioned above, one of the main results, Theorem \ref{jan.17gf},
of this paper will show that, if $R$ is an excellent local ring
which satisfies the condition $(R_0)$, then $R$ has a test element
for ideals.

Establishment of the existence of test elements for $R$ in the case
where it is not local (but is excellent and satisfies condition
$(R_0)$) appears to present additional hurdles. The main result of
this paper in the not-necessarily-local case concerns the situation
where $R$ is a homomorphic image of an excellent regular ring $S$ of
characteristic $p$ for which the Frobenius homomorphism $f : S \lra
S$ is `intersection flat' (see \cite[p.\ 41]{HocHun94} or
Definitions \ref{if.1} for the definition of this concept). Examples
of regular rings of characteristic $p$ for which the Frobenius
homomorphism is intersection flat include $\K[X_1,\ldots,X_n]$ and
$\K[[X_1,\ldots,X_n]]$, where $\K$ is a field of characteristic $p$
and $X_1,\ldots,X_n$ are indeterminates, and $
\K[[X_1,\ldots,X_n]][Y_1,\ldots,Y_m]$, where $X_1,\ldots,X_n,Y_1,
\ldots,Y_m$ are independent indeterminates and the field $\K$ of
characteristic $p$ is such that $\K^{1/p}$ is a finite extension of
$\K$.

The following result is presented in Theorem \ref{nlrn.2} below.

\vspace{0.1in}

\noindent{\bf Theorem.} {\it Suppose that $R$ is a homomorphic image
$S/\fa$ of an excellent regular ring $S$ of characteristic $p$
modulo a proper ideal $\fa$, and assume that $f : S \lra S$ is
intersection-flat and that $R$ satisfies condition $(R_0)$. Then for
each $c \in R^{\circ}$ for which $R_c$ is Gorenstein and weakly
$F$-regular, some power of $c$ is a big test element for $R$.}

\vspace{0.1in}

This result is not covered by the Hochster--Huneke Theorem
\ref{in.1} because it applies to rings (of a certain type) that
satisfy condition $(R_0)$ and it asserts the existence of big test
elements (although, as mentioned earlier, Hochster and Huneke are
aware that their methods yield big test elements). The following
corollary (presented below as Corollary \ref{jul.5}) extends the
applicability of the circle of ideas underlying Theorem
\ref{nlrn.2}.

\vspace{0.1in}

\noindent{\bf Corollary.} {\it Suppose that $R \subseteq R'$ is a
faithfully flat extension of excellent rings of characteristic $p$
such that all the fibre rings of the inclusion ring homomorphism are
regular, and such that $R'$ is a homomorphic image of an excellent
regular ring $S$ of characteristic $p$ for which the Frobenius
homomorphism $f : S \lra S$ is intersection-flat.

Suppose that $R$ satisfies condition $(R_0)$. Then for each $c \in
R^{\circ}$ for which $R_c$ is Gorenstein and weakly $F$-regular,
some power of $c$ is a big test element for $R$.}

\vspace{0.1in}

The following strengthening of Theorem \ref{in.11} is another new
result in the not-necessarily-local case that is presented in this
paper (in Corollary \ref{may.2}).

\vspace{0.1in}

\noindent{\bf Theorem.} {\it Suppose that $R$ is excellent and
$F$-pure (but not necessary local), and that $c \in R^{\circ}$ is
such that $R_c$ is Gorenstein and weakly $F$-regular. Then $c$
itself is a big test element for $R$.}

\vspace{0.1in}

The methods used in this paper are based on the theory of left
$R[x,f]$-modules and present an approach to questions about
existence of test elements that is very different from that of
Hochster and Huneke in \cite{HocHun94}. Indeed, one of the main aims
of this paper is to provide arguments that establish the existence
of big test elements without use of the Gamma construction.

I would like to thank the referee for several very helpful comments.

\section{\it Annihilators of internal tight closures of zero}
\label{bt}

The notation and terminology introduced in \S \ref{intro} will be
maintained throughout the paper.

\begin{ntn}
\label{jan.21} Throughout the paper, $H$ will denote a left
$R[x,f]$-module; its {\em $x$-torsion submodule\/} is its
$R[x,f]$-submodule
$$
\Gamma_x(H) := \left\{ h \in H : \text{there exists $n \in \nn$ such
that $x^nh = 0$}\right\}.
$$
We shall denote the left $R[x,f]$-module $H/\Gamma_x(H)$ by $G$.
Note that $G$ is {\em $x$-torsion-free\/}, that is, $\Gamma_x(G) =
0$.

The graded two-sided ideals of $R[x,f]$ are precisely the subsets of
the form $${\textstyle \bigoplus_{n\in\nn}\fa_nx^n},$$ where
$(\fa_n)_{n\in\nn}$ is an ascending sequence of ideals of $R$. (Of
course, such a sequence $\fa_0 \subseteq \fa_1 \subseteq \cdots
\subseteq \fa_n \subseteq \cdots $ is eventually stationary.) An
$R[x,f]$-submodule of $H$ is said to be a {\em special annihilator
submodule of $H$\/} if it has the form $$\ann_H(\fA) = \{ h \in H :
\theta h = 0 \text{~for all~} \theta \in \fA\}$$ for some {\em
graded\/} two-sided ideal $\fA$ of $R[x,f]$.

We shall use $\mathcal{A}(H)$ to denote the set of special
annihilator submodules of $H$.

The {\em graded annihilator $\grann_{R[x,f]}H$ of $H$\/} is defined
to be the largest graded two-sided ideal of $R[x,f]$ that
annihilates $H$. Note that, if $\grann_{R[x,f]}H =
\bigoplus_{n\in\nn} \fa_nx^n$, then $\fa_0 = (0:_RH)$, which we
shall sometimes refer to as the\/ {\em $R$-annihilator of $H$.}

We shall use $\mathcal{I}(H)$ to denote the set of $R$-annihilators
of $R[x,f]$-submodules of $H$.
\end{ntn}

We now recall some satisfactory properties of the set
$\mathcal{I}(H)$ in the case where $H$ is $x$-torsion-free, that is,
when $H = G$.

\begin{rmds}
\label{jan.23} Consider the $x$-torsion-free left $R[x,f]$-module
$G$.

\begin{enumerate} \item By \cite[Lemma 1.9]{ga}, the members of $\mathcal{I}(G)$ are all
radical ideals of $R$; they are referred to as the {\em $G$-special
$R$-ideals\/}; in fact, the graded annihilator of an
$R[x,f]$-submodule $L$ of $G$ is equal to $\bigoplus_{n\in\nn}\fa
x^n$, where $\fa = (0:_RL)$.

\item By \cite[Proposition 1.11]{ga}, there is an order-reversing bijection, $\Theta : \mathcal{A}(G)
\lra \mathcal{I}(G)$ given by
$$
\Theta : N \longmapsto \left(\grann_{R[x,f]}N\right)\cap R =
(0:_RN).
$$
The inverse bijection, $\Theta^{-1} : \mathcal{I}(G) \lra
\mathcal{A}(G),$ also order-reversing, is given by
$$
\Theta^{-1} : \fb \longmapsto \ann_G(\fb R[x,f]).
$$

\item By \cite[Corollary 3.7]{ga}, the set of $G$-special
$R$-ideals is precisely the set of all finite intersections of prime
$G$-special $R$-ideals (provided one includes the empty
intersection, $R$, which corresponds under the bijection of part
(ii) to the zero special annihilator submodule of $G$). In symbols,
$$
\mathcal{I}(G) = \left\{ \fp_1 \cap \ldots \cap \fp_t : t \in \nn
\mbox{~and~} \fp_1, \ldots, \fp_t \in
\mathcal{I}(G)\cap\Spec(R)\right\}.
$$

\item By \cite[Corollary 3.11]{ga}, when $G$ is Artinian as an
$R$-module, the sets $\mathcal{I}(G)$ and $\mathcal{A}(G)$ are both
finite.
\end{enumerate}
\end{rmds}

\begin{ntn}
\label{jan.24} We define $\mathcal{M}(G)$, for the $x$-torsion-free
left $R[x,f]$-module $G$, to be the set of minimal members (with
respect to inclusion) of the set $$ \left\{ \fp \in \Spec(R) \cap
\mathcal{I}(G) : \height \fp
> 0\right\} $$ of prime $G$-special $R$ ideals of positive height.

When $\mathcal{M}(G)$ is finite, we shall set $\fb := \bigcap_{\fp
\in \mathcal{M}(G)}\fp$; in that case, it follows from
\ref{jan.23}(iii) that $\fb$ is the smallest member of
$\mathcal{I}(G)$ of positive height (and, in particular, $\fb$ is
contained in every other member of $\mathcal{I}(G)$ of positive
height). (In the special case where $\mathcal{M}(G) = \emptyset$, we
interpret $\fb$ as $R$, the intersection of the empty family of
prime ideals of $R$. We adopt the convention that the improper ideal
$R$ has infinite (and positive) height.)
\end{ntn}

In this section we shall explore the concept of the internal tight
closure of zero in the left $R[x,f]$-module $H$, defined as follows.

\begin{defi}
\label{jan.3} We define the {\em internal tight closure of zero in
$H$,\/} denoted $\Delta(H)$ (or $\Delta_R(H)$ when it is desirable
to emphasize the underlying ring $R$), to be the $R[x,f]$-submodule
of $H$ given by
$$
\Delta(H)= \left\{ h \in H : \text{there exists~} c \in R^{\circ}
\text{~with~} cx^nh = 0 \text{~for all~} n \gg 0 \right\}.
$$
Note that if the left $R[x,f]$-module $H$ is $\Z$-graded, then
$\Delta(H)$ is a graded submodule.
\end{defi}

Connections between internal tight closures of zero and tight
closures are given in the following two examples.

\begin{ex}
\label{jan.5} Let $M$ be an $R$-module and let $h,n \in \nn$. Endow
$Rx^n$ and $Rx^h$ with their natural structures as $(R,R)$-bimodules
(inherited from their being graded components of $R[x,f]$). Then
there is an isomorphism of (left) $R$-modules $\phi :
Rx^{n+h}\otimes_R M \stackrel{\cong}{\lra}
Rx^{n}\otimes_R(Rx^{h}\otimes_R M)$ for which $\phi(rx^{n+h} \otimes
m) = rx^n \otimes (x^h \otimes m)$ for all $r \in R$ and $m \in M$.

One can use isomorphisms like that described in the above paragraph
to see that
$$\Delta(R[x,f]\otimes_RM) = 0^*_M \oplus 0^*_{Rx
\otimes_RM} \oplus \cdots \oplus 0^*_{Rx^n \otimes_RM}\oplus \cdots.
$$
\end{ex}

\begin{ex}
\label{jan.5jalg} Suppose that $(R,\fm)$ is local of dimension $d >
0$, and endow the local cohomology module $H' := H^d_{\fm}(R)$ with
its natural structure as a left $R[x,f]$-module (see \cite[Reminder
4.1]{ga}, for example). Then $\Delta(H')= 0^*_{H'}$, by \cite[Remark
4.2(iv)]{ga}.
\end{ex}

\begin{rmk}
\label{jan.2} Motivated by tight closure theory, we shall be
interested in situations where the ideal $(0:_R\Delta(H))$ has
positive height. In this context, it should be noted that an ideal
$\fa$ of $R$ has positive height if and only if it meets
$R^{\circ}$, and then it can be generated by (finitely many)
elements of $\fa\cap R^{\circ}$.
\end{rmk}

\begin{rmk}\label{jan.21a}
If $\fA = \bigoplus_{n\in\nn}\fa_nx^n$ is a graded two-sided ideal
of $R[x,f]$, so that $$\fa_0 \subseteq \fa_1 \subseteq \cdots
\subseteq \fa_n \subseteq \cdots$$ is an ascending chain of ideals
of $R$, and $\height \fa_t
> 0$ for some $t \in \nn$, then $\ann_H(\fA) \subseteq \Delta(H)$
because there exists $c \in \fa_t \cap R^{\circ}$, so that each
element of $\ann_H(\fA)$ is annihilated by $cx^j$ for all $j \geq
t$.
\end{rmk}

\begin{rmk}\label{jan.4}
Note that
\begin{align*}
\Gamma_x(H) & = \left\{h\in H : \text{there exists~} n \in \nn
\text{~with~} x^nh = 0 \right\}
\\ &= \left\{ h \in H : 1x^nh = 0 \text{~for all~} n \gg 0 \right\} \subseteq
\Delta(H).
\end{align*}
Therefore $(0:_R\Delta(H)) \subseteq (0:_R\Gamma_x(H))$.
\end{rmk}

\begin{lem}
\label{jan.25} Let $\grann_{R[x,f]}(\Delta(H)) :=
\bigoplus_{n\in\nn}\fd_nx^n$, so that $\fd_0 = (0:_R\Delta(H))$.
Suppose that $\height \fd_{n_0} > 0$ for some $n_0\in \nn$. Then
$\Delta(H) = {\textstyle \ann_H(\bigoplus_{n\geq n_0}\fd_{n_0}x^n)}=
{\textstyle \ann_H(\bigoplus_{n\geq n_0}\fd_nx^n)}$. In particular,
if\/ $\height \fd_0 > 0$, then
$$
\Delta(H) = \ann_H(\fd_0 R[x,f]) = {\textstyle
\ann_H(\bigoplus_{n\in\nn}\fd_nx^n)}.
$$
\end{lem}

\begin{note} It should be noted that, in the case where $\height \fd_0 > 0$,
it is {\em not\/} being claimed
in this lemma that $\fd_n = \fd_0$ for all $n \in \nn$: recall that
$H$ need not be $x$-torsion-free.
\end{note}

\begin{proof} Since $\fd_n \subseteq \fd_{n+1}$ for all $n \in \nn$, it is clear that
$$
\Delta(H) \subseteq {\textstyle \ann_H(\bigoplus_{n\geq
n_0}\fd_{n}x^n)} \subseteq {\textstyle \ann_H(\bigoplus_{n\geq
n_0}\fd_{n_0}x^n)}).
$$
However, since $\height \fd_{n_0} > 0$, it follows from Remark
\ref{jan.21a} that $\ann_H(\bigoplus_{n\geq n_0}\fd_{n_0}x^n)
\subseteq \Delta(H)$.
\end{proof}

\begin{lem}
\label{jan.6} Recall that $G := H/\Gamma_x(H)$ and, from Remark\/
{\rm \ref{jan.4}}, that $\Gamma_x(H) \subseteq \Delta(H)$. Assume
that $\height (0:_R\Gamma_x(H)) > 0$. Then $\Delta(H)/\Gamma_x(H) =
\Delta(G)$.
\end{lem}

\begin{proof}
It is clear that, in any event, $\Delta(H)/\Gamma_x(H)
\subseteq \Delta(H/\Gamma_x(H))$. To prove the reverse inclusion,
let $h \in H$ be such that $h + \Gamma_x(H) \in
\Delta(H/\Gamma_x(H))$, so that there exist $c \in R^{\circ}$ and
$n_0 \in \nn$ such that $cx^n(h + \Gamma_x(H)) = 0$ for all $n \geq
n_0$. Thus  $cx^nh \in \Gamma_x(H)$ for all $n \geq n_0$. There
exists $r \in R^{\circ}\cap(0:_R\Gamma_x(H))$; therefore $rcx^nh =
0$ for all $n \geq n_0$; since $rc \in R^{\circ}$, we have $h \in
\Delta(H)$. It follows that $\Delta(H)/\Gamma_x(H) =
\Delta(H/\Gamma_x(H))$.
\end{proof}

\begin{lem}
\label{jan.7} Consider the $x$-torsion-free left $R[x,f]$-module
$G$. If $\mathcal{M}(G)$ (see\/ {\rm \ref{jan.24}}) is finite, then
$\Delta(G) = \ann_G(\fb R[f,x])$, where $\fb:= \bigcap_{\fp \in
\mathcal{M}(G)}\fp$.
\end{lem}

\begin{proof} Since $\mathcal{M}(G)$ is finite, $\fb$ is the smallest
member of ${\mathcal I}(G)$ of positive height. By Remark
\ref{jan.21a}, we have $\ann_G(\fb R[x,f]) \subseteq \Delta(G)$. Let
$g \in \Delta(G)$, so that there exist $c' \in R^{\circ}$ and $n_0
\in\nn$ such that $c'x^ng = 0$ for all $n \geq n_0$. Thus $g \in
\ann_G(\bigoplus_{n\geq n_0}Rc'x^n) =: J$, a special annihilator
submodule of $G$. Let $\fb'$ be the $G$-special $R$-ideal that
corresponds to this special annihilator submodule (in the bijective
correspondence of \ref{jan.23}(ii)). Since $\bigoplus_{n\geq
n_0}Rc'x^n \subseteq \grann_{R[x,f]}J = \fb' R[x,f]$, we have $c'
\in \fb'$, so that $\height \fb' > 0$. Therefore $\fb' \supseteq
\fb$, and $g \in J = \ann_G(\fb' R[x,f]) \subseteq \ann_G(\fb
R[x,f]). $ Therefore $\Delta(G) \subseteq \ann_G(\fb R[x,f])$. We
conclude that $\Delta(G) = \ann_G(\fb R[x,f])$.
\end{proof}

\begin{thm}
\label{jan.8} Use the notation of\/ {\rm \ref{jan.21}}: recall that
$G = H/\Gamma_x(H)$. The following statements are equivalent:
\begin{enumerate}
\item $\height
(0:_R\Delta(H)) > 0$;
\item $\height (0:_R\Gamma_x(H)) > 0$ and $\height (0:_R\Delta(G)) >
0$;
\item $\height (0:_R\Gamma_x(H)) > 0$ and $\mathcal{M}(G)$ is finite.
\end{enumerate}
When these conditions are satisfied, $(0:_R\Delta(G)) = \fb :=
{\textstyle \bigcap_{\fp \in \mathcal{M}(G)}\fp}$ and
$$(0:_R\Delta(G))(0:_R\Gamma_x(H)) \subseteq (0:_R\Delta(H))
\subseteq (0:_R\Delta(G))\cap (0:_R\Gamma_x(H)).$$ Consequently, if
$n\in \nn$ is such that $\left(\sqrt{(0:_R\Gamma_x(H))}\right)^n
\subseteq (0:_R\Gamma_x(H))$, then
$$\left(\sqrt{(0:_R\Delta(H))}\right)^{n+1} \subseteq
(0:_R\Delta(H)).$$
\end{thm}

\begin{proof} (i) $\Rightarrow$ (ii) Assume that $\height (0:_R\Delta(H)) > 0$. Now
$(0:_R\Delta(H)) \subseteq (0:_R\Gamma_x(H))$ by Remark \ref{jan.4},
so that $\height (0:_R\Gamma_x(H)) > 0$. By Lemma \ref{jan.6}, we
have $$\Delta(H)/\Gamma_x(H) = \Delta(G).$$ Therefore $
(0:_R\Delta(H)) \subseteq (0:_R\Gamma_x(H))\cap (0:_R\Delta(G)). $
It follows that $\height (0:_R\Delta(G)) \geq \height
(0:_R\Delta(H))
> 0$.

(ii) $\Rightarrow$ (iii) Assume that $\height (0:_R\Gamma_x(H)) > 0$
and $\height (0:_R\Delta(G)) > 0$, and set $\fc := (0:_R\Delta(G))$.
Since $G$ is $x$-torsion-free and $\Delta(G)$ is an
$R[x,f]$-submodule of $G$, it follows from \ref{jan.23}(i) that
$\fc$ is radical and $\grann_{R[x,f]}\Delta(G) = \fc R[x,f]$. Since
$\height \fc > 0$, we have $\ann_{G}(\fc R[x,f]) = \Delta(G)$ by
Lemma \ref{jan.25}. Note that $\fc \in \mathcal{I}(G)$.

If $\fc = R$, then $\Delta(G) = 0$, so that $\mathcal{M}(G)$ is
empty because a $\fp \in \mathcal{I}(G)\cap\Spec(R)$ with $\height
\fp > 0$ must (by Remark \ref{jan.21a}) satisfy $\ann_G(\fp R[x,f])
\subseteq \Delta(G) = 0$, and this leads to a contradiction to
\ref{jan.23}(ii). We therefore suppose that $\fc \neq R$.

Let $\fc = \fp_1 \cap \cdots \cap \fp_t$ be the minimal primary
decomposition of the (radical) ideal $\fc$. By \cite[Theorem 3.6 and
Corollary 3.7]{ga}, the prime ideals $\fp_1, \ldots, \fp_t$ all
belong to $\mathcal{I}(G)$; they all have positive height. Since
$\Spec (R)$ satisfies the descending chain condition, each member of
$\{\fp' \in \Spec (R)\cap\mathcal{I}(G) : \height \fp'
> 0\}$ contains a member of ${\mathcal M}(G)$. In particular, each of $\fp_1, \ldots, \fp_t$ contains
a member of ${\mathcal M}(G)$.

Now let $\fp \in {\mathcal M}(G)$. We use Remark \ref{jan.21a} to
see that $\ann_{G}(\fp R[x,f]) \subseteq \Delta(G) = \ann_{G}(\fc
R[x,f])$. It now follows from the inclusion-reversing bijective
correspondence of \ref{jan.23}(ii) that $\fp \supseteq \fc$, so that
$\fp$ contains one of $\fp_1, \ldots, \fp_t$. We can therefore
conclude that $\fp_1, \ldots, \fp_t$ are precisely the minimal
members of $\{\fp \in \mathcal{I}(G)\cap\Spec (R) : \height \fp
> 0\}$. Therefore $\mathcal{M}(G)$ is finite.

(iii) $\Rightarrow$ (i) Assume that $\height (0:_R\Gamma_x(H)) > 0$
and that $\mathcal{M}(G)$ is finite. It follows from \ref{jan.24}
that $\fb := \bigcap_{\fp \in \mathcal{M}(G)}\fp$ is the smallest
member of $\{ \fh \in \mathcal{I}(G) : \height \fh
> 0 \}$ and Lemma \ref{jan.7} shows that $\Delta(G) =
\ann_{G}(\fb R[x,f])$. Therefore $(0:_R\Delta(G)) = \fb$ has
positive height. Lemma \ref{jan.6} shows that $\Delta(H)/\Gamma_x(H)
= \Delta(G)$, so that
$$
(0:_R\Gamma_x(H))(0:_R\Delta(G))\subseteq (0:_R\Delta(H)).
$$
Since $\height (0:_R\Gamma_x(H)) > 0$ and $\height (0:_R\Delta(G))>
0$, we see that $\height (0:_R\Delta(H))> 0$.

It follows from what we have already established in this proof that
if the conditions (i), (ii) and (iii) are satisfied, then
$(0:_R\Delta(G)) = \fb := {\textstyle \bigcap_{\fp \in
\mathcal{M}(H)}\fp}$ and
$$(0:_R\Delta(G))(0:_R\Gamma_x(H)) \subseteq (0:_R\Delta(H))
\subseteq (0:_R\Delta(G))\cap (0:_R\Gamma_x(H)).$$ Finally, suppose
that $n\in \N$ is such that $\left(\sqrt{(0:_R\Gamma_x(H))}\right)^n
\subseteq (0:_R\Gamma_x(H))$. Then
\begin{align*}
\left(\sqrt{(0:_R\Delta(H))}\right)^{n+1} & =
\sqrt{(0:_R\Delta(H))}\left(\sqrt{(0:_R\Delta(H))}\right)^{n}\\
&\subseteq
\sqrt{(0:_R\Delta(G))}\left(\sqrt{(0:_R\Gamma_x(H))}\right)^n\\
&\subseteq (0:_R\Delta(G))(0:_R\Gamma_x(H)) \subseteq
(0:_R\Delta(H))
\end{align*}
because $(0:_R\Delta(G))$ is a radical ideal.
\end{proof}

\begin{cor}
\label{jan.26} Use the notation of\/ {\rm \ref{jan.21}} and suppose
that $H$ is Artinian as an $R$-module. Then $\height (0:_R\Delta(H))
> 0$ if and only if $\height (0:_R\Gamma_x(H)) > 0$.
\end{cor}

\begin{proof} Since $H$ is Artinian as an $R$-module, so too is $G$.
Therefore $\mathcal{I}(G)$ is finite, by \ref{jan.23}(iv). It
follows that $\mathcal{M}(G)$ is automatically finite, and so the
claim follows from the equivalence of (i) and (iii) in Theorem
\ref{jan.8}.
\end{proof}

In the case where $(R,\fm)$ is local and complete and satisfies the
condition $(R_0)$, we shall use Corollary \ref{jan.26} in \S\S
\ref{set},\ref{hrlr} in order to produce a left $R[x,f]$-module
structure on the injective envelope $E$ of the simple $R$-module for
which $\height (0:_R\Delta(E)) > 0$: the corollary shows that it is
enough for us to arrange that $\height (0:_R\Gamma_x(E)) > 0$.

\section{\it Invariance of the ideal $(0:_R\Delta(H))$}
\label{inv}

For reasons related to tight closure, we are going to be interested
in left $R[x,f]$-modules $H$ that satisfy the three equivalent
conditions of Theorem \ref{jan.8}. We shall wish to perform various
operations and constructions on such $H$s, and the purpose of this
section is to show that the operations we have in mind do not
destroy the properties described in \ref{jan.8}(i),(ii),(iii). The
key point is that the ideal $(0:_R\Delta(H))$ is left unchanged by
the operations.

\begin{ex}
\label{jan.10} Use the notation of\/ {\rm \ref{jan.21}}. For all $n
\in \nn$, set $H_n := H$. Then the $R$-module $\widetilde{H} :=
\bigoplus_{n\in\nn} H_n$ has a natural structure as a graded left
$R[x,f]$-module under which the result of multiplying $h_n \in H_n =
H$ on the left by $x$ is the element $xh_n \in H_{n+1} = H$. We
refer to $\widetilde{H}$ as the {\em graded companion of $H$}.

It is easy to check that $\Delta(\widetilde{H}) =
\widetilde{\Delta(H)}$, so that $(0:_R\Delta(\widetilde{H})) =
(0:_R\Delta(H))$, that $\Gamma_x(\widetilde{H}) =
\widetilde{\Gamma_x(H)}$, so that $(0:_R\Gamma_x(\widetilde{H})) =
(0:_R\Gamma_x(H))$, and that, if $\fA$ is a graded two-sided ideal
of $R[x,f]$, then $(0:_R\ann_{\widetilde{H}}\fA) = (0:_R\ann_H\fA)$.
\end{ex}

\begin{ex}
\label{jan.11} Let $(H^{(\lambda)})_{\lambda\in\Lambda}$ be a
non-empty family of $\Z$-graded left $R[x,f]$-modules, with gradings
given by $H^{(\lambda)}= \bigoplus_{n\in\Z}H^{(\lambda)}_n$ for each
$\lambda\in\Lambda$. The $R$-module
\[
\prod_{\lambda\in\Lambda}{\textstyle ^{^{^{\Large
\prime}}}}H^{(\lambda)} :=
\bigoplus_{n\in\Z}\left(\prod_{\lambda\in\Lambda}H^{(\lambda)}_n\right)
\]
is a ($\Z$-graded) left $R[x,f]$-module in which
\[x\big(h^{(\lambda)}_n\big)_{\lambda\in\Lambda} =
\big(xh^{(\lambda)}_n\big)_{\lambda\in\Lambda}\in
\prod_{\lambda\in\Lambda}H^{(\lambda)}_{n+1} \quad \mbox{for all~}
\big(h^{(\lambda)}_n\big)_{\lambda\in\Lambda} \in
\prod_{\lambda\in\Lambda}H^{(\lambda)}_n.\] See \cite[Lemma
2.1]{gatcti}. In this paper, we shall refer to
$\prod^{\prime}_{\lambda\in\Lambda}H^{(\lambda)}$ as the {\em graded
product\/} of the $H^{(\lambda)}$.

It is routine to check that, if there is an ideal $\fd_0$ of $R$
such that $(0:_R\Delta(H^{(\lambda)})) = \fd_0$ for all
$\lambda\in\Lambda$, then $\left( 0:_R
\Delta\left(\prod^{\prime}_{\lambda\in\Lambda}H^{(\lambda)}\right)\right)
= \fd_0$. Also, if there is an ideal $\fa_0$ of $R$ such that
$(0:_R\Gamma_x(H^{(\lambda)})) = \fa_0$ for all $\lambda\in\Lambda$,
then $\left( 0:_R
\Gamma_x\!\left(\prod^{\prime}_{\lambda\in\Lambda}H^{(\lambda)}\right)\right)
= \fa_0$.

Let $\fA$ be a graded two-sided ideal of $R[x,f]$. If there is an
ideal $\fb_0$ of $R$ such that $(0:_R\ann_{H^{(\lambda)}}\fA) =
\fb_0$ for all $\lambda\in\Lambda$, then
$\left(0:_R\ann_{\prod^{\prime}_{\lambda\in\Lambda}H^{(\lambda)}}\fA\right)
= \fb_0$ (by \cite[Remark 2.2]{gatcti}).
\end{ex}

\begin{ex}
\label{jan.12} Let $H = \bigoplus_{n\in\Z}H_n$ be a $\Z$-graded left
$R[x,f]$-module and let $t \in \Z$. The result of application of the
{\em $t$th shift functor\/} to $H$ is $H(t) =
\bigoplus_{n\in\Z}H(t)_n$, where $H(t)_n = H_{t+n}$ for all $n \in
\Z$. It is clear that $\Delta(H(t)) = \Delta(H)(t)$, so that
$$(0:_R\Delta(H(t))) = (0:_R\Delta(H)),$$ and similarly that
$(0:_R\Gamma_x(H(t))) = (0:_R\Gamma_x(H))$, and that, for a graded
two-sided ideal $\fA$ of $R[x,f]$, we have $(0:_R\ann_{H(t)}\fA) =
(0:_R\ann_H\fA)$.
\end{ex}

The next proposition is concerned with an $h$-place extension of a
graded left $R[x,f]$-module $W := \bigoplus_{n \geq b} W_n$, where
$b, h \in \N$. The definition of this concept is given in
\cite[Definition 4.5]{Fpurhastest}, but we recall it here. Let
$(g_i)_{i\in I}$ be a family of arbitrary elements of $W_b$, and let
$V$ denote the free $R$-module $\bigoplus_{i\in I}R_i$ (where $R_i =
R$ for all $i \in I$); also, let $f : V \lra V$ be the Frobenius map
for which $f((r_i)_{i\in I}) = (r_i^p)_{i\in I}$ for all
$(r_i)_{i\in I} \in V$, and set $K := \left\{ (r_i)_{i\in I} \in V :
{\textstyle \sum_{i\in I}r_ig_i = 0}\right\}. $ The {\em $h$-place
extension $\ext(W;(g_i)_{i\in I};h)$ of $W$ by $(g_i)_{i\in I}$\/}
is the graded left $R[x,f]$-module
\[ \left(V/f^{-h}(K)\right)\oplus \cdots \oplus
\left(V/f^{-1}(K)\right)\oplus W_b \oplus  \cdots \oplus W_t \oplus
\cdots
\]
which has $W$ as a graded $R[x,f]$-submodule and is such that
$$x((r_i)_{i\in I} + f^{-1}(K)) = {\textstyle \sum_{i\in I}r_i^pg_i} \quad
\text{for all~} (r_i)_{i\in I} \in V$$ and $x(v+ f^{-j}(K)) = f(v)+
f^{-(j-1)}(K)$ for all $v \in V$ and $j = h,h-1, \ldots, 2.$

We shall now examine what happens to the ideals $(0:_R\Delta(W))$,
$(0:_R\Gamma_x(W))$ and $(0:_R\ann_W\fA)$, where $\fA$ is a graded
two-sided ideal of $R[x,f]$, when we pass from $W$ to the $h$-place
extension $\ext(W;(g_i)_{i\in I};h)$.

\begin{prop}
\label{jan.13} Let the notation be as in the above two paragraphs.
Then
\begin{enumerate}
\item $(0:_R\Delta(\ext(W;(g_i)_{i\in I};h))) = (0:_R\Delta(W))$,
\item $(0:_R\Gamma_x(\ext(W;(g_i)_{i\in I};h))) =
(0:_R\Gamma_x(W))$, and
\item $(0:_R\ann_{\ext(W;(g_i)_{i\in I};h)}\fA) = (0:_R\ann_W\fA)$
for each graded two-sided ideal $\fA$ of $R[x,f]$.
\end{enumerate}
\end{prop}

\begin{proof} We can construct the $h$-place extension $\ext(W;(g_i)_{i\in
I};h)$ of\/ $W$ by a sequence of $h$ $1$-place extensions. See
\cite[Definition 4.5]{Fpurhastest}. Thus it is enough to prove this
result in the case where $h = 1$, and this is what we shall do. Set
$W' := \ext(W;(g_i)_{i\in I};1)$.

For $n \in \Z$, we shall denote the $n$th component of a $\Z$-graded
left $R[x,f]$-module $G$ by $G_n$. We begin by showing that if $L' =
\bigoplus_{n\geq b-1}L'_n$ is a graded $R[x,f]$-submodule of $W'$
and $L := \bigoplus_{n\geq b}L'_n$, then $(0:_RL') = (0:_RL)$. It is
clear that $(0:_RL') \subseteq (0:_RL)$, because $L$ is an
$R[x,f]$-submodule of $L'$. Let $a \in (0:_RL)$; then $aL'_n = 0$
for all $n \geq b$, and we now show that $aL'_{b-1} = 0$ also. Let
$w \in L'_{b-1}$, so that $w = (r_i)_{i\in I} + f^{-1}(K)$ for some
$(r_i)_{i\in I} \in V$. Then $xw \in L'_b \subseteq L$, so that $axw
= 0$; therefore $xaw = a^pxw = 0$, so that $\sum_{i\in I}a^pr_i^pg_i
= 0$. Therefore $(ar_i)_{i\in I} \in f^{-1}(K)$ and $aw = 0$. It
follows that $(0:_RL') \supseteq (0:_RL)$, so that $(0:_RL') =
(0:_RL)$.

(i) If we choose $L'$ in the above paragraph to be the graded
$R[x,f]$-submodule $\Delta(W')$ of $W'$, then $L$ becomes
$\Delta(W)$, and the paragraph shows that $(0:_R\Delta(W)) =
(0:_R\Delta(W')) = (0:_R\Delta(\ext(W;(g_i)_{i\in I};1)))$.

(ii) Similarly, if we choose $L'$ in the second paragraph of this
proof to be the graded $R[x,f]$-submodule $\Gamma_x(W')$ of $W'$,
then $L$ becomes $\Gamma_x(W)$, and the paragraph shows that
$(0:_R\Gamma_x(W)) = (0:_R\Gamma_x(W')) =
(0:_R\Gamma_x(\ext(W;(g_i)_{i\in I};1)))$.

(iii) If we choose $L'$ in the second paragraph of this proof to be
the graded $R[x,f]$-submodule $\ann_{W'}\fA$ of $W'$, then $L$
becomes $\ann_{W}\fA$, and the paragraph shows that
$(0:_R\ann_{W}\fA) = (0:_R\ann_{W'}\fA) =
(0:_R\ann_{\ext(W;(g_i)_{i\in I};1)}\fA)$.
\end{proof}

We adopt the convention that a graded left $R[x,f]$-module $W$ as
above is considered as a $0$-place extension of itself.

\section{\it An embedding theorem}
\label{set}

In this section we shall present an embedding theorem that is
similar to, but different from, the embedding theorems of \cite[\S
4]{Fpurhastest}. Here we base the argument on the direct product of
a family of copies of an injective cogenerator of $R$, and this
approach facilitates a more rapid development than that in \cite[\S
4]{Fpurhastest}. Recall that an {\em injective cogenerator of $R$\/}
is an injective $R$-module $E$ with the property that, for every
$R$-module $M$ and every non-zero element $m \in M$, there exists an
$R$-homomorphism $\phi : M \lra E$ such that $\phi(m) \neq 0$. As
$R$ is Noetherian, $\bigoplus_{\fm \in \Max(R)} E_R(R/\fm)$, where
$\Max(R)$ denotes the set of maximal ideals of $R$, is one injective
cogenerator of $R$.

\begin{lem}
\label{pl.5a} Let $E$ be an injective cogenerator of $R$. Assume
that there exists an $\nn$-graded left $R[x,f]$-module $H =
\bigoplus_{n\in\nn}H_n$ such that $H_0$ is $R$-isomorphic to $E$.

Let $M$ be a non-zero $R$-module, and let $J := M \setminus \{0\}$,
the set of non-zero elements of $M$.  For each $j \in J$, set
$H^{(j)} := H$. Then there exists a homogeneous
$R[x,f]$-homomorphism
\[
\lambda : R[x,f]\otimes_RM \lra \prod_{j\in J}{\textstyle
^{^{^{\Large \prime}}}}H^{(j)} =: L
\]
such that its component $\lambda_0$ of degree zero is a
monomorphism.
\end{lem}

\begin{proof} Let $J := M \setminus \{0\}$. For each $m \in J$ there
exists an $R$-homomorphism $\phi_m : M \lra E$ such that $\phi_m(m)
\neq 0$. Define $\lambda_0 : M \lra L_0$ by $\lambda_0(g) =
(\phi_m(g))_{m\in J}$ for all $g \in M$. Then $\lambda_0$ is an
$R$-monomorphism. (This argument is taken from the proof of Sharpe
and V\'amos in \cite[Proposition 2.25]{SV}.)

We can then define, for each $n \in \N$, an $R$-homomorphism
$\lambda_n: Rx^n \otimes_RM \lra L_n$ for which $\lambda_n(rx^n
\otimes m) = rx^n\lambda_0(m)$ for all $r \in R$ and all $m \in M$.
It is straightforward to check that the $\lambda_n~(n\in\nn)$
provide a homogeneous $R[x,f]$-homomorphism as claimed.
\end{proof}

\begin{thm}
\label{pl.5b} Let $E$ be an injective cogenerator of $R$. Assume
that there exists an $\nn$-graded left $R[x,f]$-module $H =
\bigoplus_{n\in\nn}H_n$ such that $H_0$ is $R$-isomorphic to $E$.
Let $\fA$ be a graded two-sided ideal of $R[x,f]$.

Let $M$ be an $R$-module. Then there is a family
$\left(L^{(n)}\right)_{n \in \nn}$ of\/ $\nn$-graded left
$R[x,f]$-modules, where $L^{(n)}$ is an $n$-place extension of the
$-n$th shift of a graded product of copies of $H$ (for each $n \in
\nn$), for which there exists a homogeneous $R[x,f]$-monomorphism
\[
\nu : R[x,f]\otimes_RM = \bigoplus_{i\in \nn}(Rx^i\otimes_RM) \lra
\prod_{n\in\nn}{\textstyle ^{^{^{\!\!\Large \prime}}}} L^{(n)} =: K.
\]
Consequently
\begin{enumerate}
\item $(0:_R\Gamma_x(H)) = (0:_R\Gamma_x(K)) \subseteq
(0:_R\Gamma_x(R[x,f]\otimes_RM))$;
\item $(0:_R\Delta(H)) = (0:_R\Delta(K)) \subseteq
(0:_R\Delta(R[x,f]\otimes_RM))$;
\item $(0:_R\ann_H\fA) = (0:_R\ann_K\fA) \subseteq
(0:_R\ann_{R[x,f]\otimes_RM}\fA)$;
\item if\/ $\height (0:_R\Delta(H)) > 0$, then
$(0:_R\Delta(H))\cap R^{\circ}$ consists of big test elements for
$R$, so that, in particular, $R$ has big test elements.
\end{enumerate}
\end{thm}

\begin{proof} Let $\fa_0 := (0:_R\Gamma_x(H))$, let $\fd_0 := (0:_R\Delta(H))$
and let $\fb_0 := (0:_R\ann_H\fA)$.
Let $n \in \nn$. By Lemma \ref{pl.5a}, there is a family
$(G^{(j,n)})_{j \in Y_n}$ of graded left $R[x,f]$-modules, all equal
to $H$, and a homogeneous $R[x,f]$-homomorphism
$$R[x,f]\otimes_R(Rx^n\otimes_R M) \lra \prod_{j\in
Y_n}{\textstyle ^{^{^{\Large \prime}}}}G^{(j,n)}$$ which is
monomorphic in degree $0$. (This is trivially true if $Rx^n\otimes_R
M$ happens to be zero.) If we now use isomorphisms of the type
described in Example \ref{jan.5}, we obtain (after application of
the shift functor $ (\: {\scriptscriptstyle \bullet} \:)(-n)$) a
homogeneous $R[x,f]$-homomorphism
$$
\zeta^{(n)} : \bigoplus_{j \geq n}(Rx^j\otimes_RM) \lra
\left(\prod_{j\in Y_n}{\textstyle ^{^{^{\Large
\prime}}}}G^{(j,n)}\right)(-n)
$$
which is monomorphic in degree $n$. We can now use \cite[Proposition
4.6]{Fpurhastest} to extend $\zeta^{(n)}$ by $n$ places to produce a
homogeneous $R[x,f]$-homomorphism
$$
\lambda^{(n)} : \bigoplus_{j \geq 0}(Rx^j\otimes_RM) =
R[x,f]\otimes_RM \lra L^{(n)},
$$
where $L^{(n)}$ is an appropriate $n$-place extension of
$\left(\prod_{j\in Y_n}^{\prime} G^{(j,n)}\right)(-n)$, such that
$\lambda^{(n)}$ is monomorphic in degree $n$.

It follows from Examples \ref{jan.11} and \ref{jan.12} and
Proposition \ref{jan.13} that $(0:_R\Gamma_x(L^{(n)})) = \fa_0$,
that $(0:_R\Delta(L^{(n)})) = \fd_0$ and that
$(0:_R\ann_{L^{(n)}}\fA) = \fb_0$ for all $n \in \nn$.

Next, we use the $\lambda^{(n)}~(n\in\nn)$ to produce a homogeneous
$R[x,f]$-homomorphism
\[
\nu  = \bigoplus_{j\in \nn}\nu_j : R[x,f]\otimes_RM \lra \prod_{
n\in\nn}{\textstyle ^{^{^{\!\!\Large \prime}}}} L^{(n)} =: K
\]
such that $\nu_j(\xi_j) =
\big((\lambda^{(n)})_j(\xi_j)\big)_{n\in\nn}$ for all $j \in \nn$
and $\xi_j \in Rx^j \otimes_RM$. For each $j \in \nn$, the map
$(\lambda^{(j)})_j$ is a monomorphism; hence $\nu_j$ is a
monomorphism. Hence $\nu$ is an $R[x,f]$-monomorphism.

Another use of Example \ref{jan.11} shows that $(0:_R\Gamma_x(K)) =
\fa_0$, that $(0:_R\Delta(K)) = \fd_0$ and that $(0:_R\ann_K\fA) =
\fb_0$.

The existence of the $R[x,f]$-monomorphism $\nu$ therefore shows
that
$$
\fd_0 = (0:_R\Delta(K)) \subseteq (0:_R\Delta(R[x,f]\otimes_RM)),
$$
that $\fa_0 \subseteq (0:_R\Gamma_x(R[x,f]\otimes_RM))$ and that
$\fb_0 \subseteq (0:_R\ann_{R[x,f]\otimes_RM}\fA)$.

Finally, suppose that $\height\fd_0 > 0$. Then there exist elements
in $\fd_0\cap R^{\circ}$; let $c$ be one such. Since $\fd_0
\subseteq (0:_R\Delta(R[x,f]\otimes_RM))$ for every $R$-module $M$,
it follows that $c$ is a big test element for $R$.
\end{proof}

\begin{thm}
\label{pl.5c} Let $E$ be an injective cogenerator of $R$. Then there
exists a big test element for $R$ if and only if\/ $\height
(0:_R\Delta(R[x,f]\otimes_RE)) > 0$. When this is the case,
$(0:_R\Delta(R[x,f]\otimes_RE))$ is the big test ideal of $R$ (that
is, the ideal of $R$ generated by all the big test elements for $R$)
and $(0:_R\Delta(R[x,f]\otimes_RE))\cap R^{\circ}$ is the set of big
test elements for $R$.
\end{thm}

\begin{proof} ($\Rightarrow$) Assume that $c \in R^{\circ}$ is a big
test element for $R$. Then, in particular, $c$ annihilates $0^*_E$
and $0^*_{Rx^h \otimes_RE}$ for all $h \in \N$. It follows from
Example \ref{jan.5} that
$$\Delta(R[x,f]\otimes_RE) = 0^*_E \oplus 0^*_{Rx \otimes_RE} \oplus
\cdots \oplus 0^*_{Rx^n \otimes_RE}\oplus \cdots.
$$
Since $c$ annihilates the right-hand side of this, $c \in
(0:_R\Delta(R[x,f]\otimes_RE))$; as $c \in R^{\circ}$, it follows
that $\height (0:_R\Delta(R[x,f]\otimes_RE)) > 0$. Note that we have
proved that every big test element for $R$ lies in
$(0:_R\Delta(R[x,f]\otimes_RE))$.

($\Leftarrow$) Assume that $\height (0:_R\Delta(R[x,f]\otimes_RE)) >
0$. Then $ (0:_R\Delta(R[x,f]\otimes_RE))$ can be generated by
(finitely many of) its members that lie in $R^{\circ}$.

We appeal to Theorem \ref{pl.5b}, with $H$ taken as
$R[x,f]\otimes_RE$; we conclude that each element of
$(0:_R\Delta(R[x,f]\otimes_RE))\cap R^{\circ}$ is a big test element
for $R$.
\end{proof}

\begin{rmk}\label{jul.1} The special case of Theorem \ref{pl.5c} in
which $(R,\fm)$ is local and Gorenstein, and $E$ is taken to be
$E_R(R/\fm)$, is worthy of additional comment. In this case, $E
\cong H^{\dim R}_{\fm}(R) =: H$, and the latter $R$-module carries a
natural structure as a left $R[x,f]$-module (as recalled in
\cite[Reminder 4.1]{ga}). By \cite[Remark 4.2(iii)]{ga}, there is an
$R$-isomorphism $Rx^n \otimes_R H \cong H$ for all $n \in \nn$, so
that $$ (0:_R\Delta(R[x,f]\otimes_RE)) = (0:_R0^*_H) = (0:_R0^*_E)$$
because $\Delta(R[x,f]\otimes_RH) = 0^*_H \oplus 0^*_{Rx \otimes_RH}
\oplus \cdots \oplus 0^*_{Rx^n \otimes_RH}\oplus \cdots $ by Example
\ref{jan.5}. G. Lyubeznik and K. E. Smith (see \cite[Theorem
8.8]{LyuSmi01}) have shown that, in a Cohen--Macaulay local ring
$(R',\fm')$ of characteristic $p$ which is Gorenstein on its
punctured spectrum, $(0:_{R'}0^*_{E_{R'}(R'/\fm')})$ is equal to the
test ideal of $R'$ \cite[Proposition (8.23)]{HocHun90}.
\end{rmk}

\begin{cor}
\label{pl.6} Assume that $(R,\fm)$ is local, and set $E :=
E_R(R/\fm)$. Suppose that $E$ can be given the structure of a left
$R[x,f]$-module that extends its structure as an $R$-module, in such
a way that $\height (0:_R\Gamma_x(E))> 0$. Then $\height
(0:_R\Delta(E)) > 0$, and $(0:_R\Delta(E)) \cap R^{\circ}$ consists
of big test elements for $R$.
\end{cor}

\begin{proof} Since $E$ is Artinian as an $R$-module,
Corollary \ref{jan.26} shows that $\height (0:_R\Delta(E)) > 0$.

We now use the graded companion $\widetilde{E}$ of $E$. This is
described in Example \ref{jan.10}, where it is pointed out that
$(0:_R\Delta(\widetilde{E}))= (0:_R\Delta(E))$.

We appeal to Theorem \ref{pl.5b}, with $H$ taken as $\widetilde{E}$;
we conclude that each element of $(0:_R\Delta(\widetilde{E}))\cap
R^{\circ} = (0:_R\Delta(E))\cap R^{\circ}$ is a big test element for
$R$.
\end{proof}

\begin{cor}
\label{jul.26} Let $E$ be an injective cogenerator of $R$. Let $n
\in \N$ and $t \in \nn$ be such that
$$\left(\sqrt{(0:_R\Delta(R[x,f]\otimes_RE))}\right)^n \subseteq
(0:_R\Delta(R[x,f]\otimes_RE))$$ and
$\left(\sqrt{(0:_R\Gamma_x(R[x,f]\otimes_RE))}\right)^t \subseteq
(0:_R\Gamma_x(R[x,f]\otimes_RE))$.

Let $c \in R^{\circ}$ be such that some power of $c$ is a big test
element for $R$. Then $c^n$ and $c^{t+1}$ are big test elements for
$R$.
\end{cor}

\begin{proof} Since $(0:_R\Delta(R[x,f]\otimes_RE))$ is the
big test ideal of $R$ (by Theorem \ref{pl.5c}) and $c$ belongs to
the radical of this ideal, $c^n$ is a big test element for $R$.

By Theorem \ref{jan.8},
$$\left(\sqrt{(0:_R\Delta(R[x,f]\otimes_RE))}\right)^{t+1} \subseteq
(0:_R\Delta(R[x,f]\otimes_RE)),$$ so that $c^{t+1}$ is also a big
test element for $R$.
\end{proof}

\section{\it Change of rings}

Some results in \S \ref{set} concern graded left $R[x,f]$-modules of
the form $R[x,f]\otimes_RE$ for injective $R$-modules $E$. Some
choices for $E$ can lead to $R[x,f]\otimes_RE$ having a natural
structure as a graded left module over $R'[x,f]$ for a different
(commutative Noetherian) ring $R'$ of characteristic $p$. For
example, if $\fp \in \Spec (R)$, then $E_R(R/\fp)$ has a natural
structure as an $R_{\fp}$-module, and this means that
$R[x,f]\otimes_RE_R(R/\fp)$ has a natural structure as a graded left
module over $R_{\fp}[x,f]$. Similarly, if $(R,\fm)$ is local, then
each element of $E_R(R/\fm)$ is annihilated by some power of $\fm$
and $E_R(R/\fm)$ has a natural structure as a module over the
completion $\widehat{R}$ of $R$; this means that
$R[x,f]\otimes_RE_R(R/\fm)$ has a natural structure as a graded left
module over $\widehat{R}[x,f]$. This section contains some `change
of rings' theorems that address, among other things, questions about
whether application of $\Delta$ over the new ring yields the same
object as application of $\Delta$ over the original ring.

\begin{prop}
\label{feb.1} Let $\fp \in \Spec (R)$ and let\/ $\phantom{}^e$ and
$\phantom{}^c$ denote extension and contraction of ideals under the
natural ring homomorphism $R \lra R_{\fp}$. Recall that $E_R(R/\fp)$
has a natural structure as an $R_{\fp}$-module, and, as such,
$E_R(R/\fp) \cong E_{R_{\fp}}(R_{\fp}/\fp R_{\fp})$. Interpret the
$R_{\fp}$-module $E_R(R/\fp)$ as $E_{R_{\fp}}(R_{\fp}/\fp R_{\fp})$.
Then
\begin{enumerate}
\item for each $n \in \nn$, the $R$-module $Rx^n \otimes E_R(R/\fp)$
has a natural structure as an $R_{\fp}$-module extending its
$R$-module structure, and $R[x,f]\otimes_R E_R(R/\fp) =
\bigoplus_{n\in\nn} Rx^n \otimes E_R(R/\fp)$ has a natural structure
as a graded left $R_{\fp}[x,f]$-module that extends its structure as
a graded left $R[x,f]$-module;
\item there is a homogeneous $R_{\fp}[x,f]$-isomorphism
$$
\omega = {\textstyle \bigoplus_{n\in\nn} \omega_n : R[x,f]\otimes_R
E_R(R/\fp)} \stackrel{\cong}{\lra}
R_{\fp}[x,f]\otimes_{R_{\fp}}E_R(R/\fp)
$$
for which $\omega_n(rx^n \otimes h) = (r/1)x^n \otimes h$ for all $h
\in E_R(R/\fp)$, $r \in R$ and $n \in \nn$;
\item $\Delta_R(R[x,f]\otimes_R E_R(R/\fp)) = \Delta_{R_{\fp}}(R[x,f]\otimes_R
E_R(R/\fp))$; and
\item we have $$(0:_R
\Delta_R(R[x,f]\otimes_R E_R(R/\fp))) = (0:_{R_{\fp}}
\Delta_{R_{\fp}}(R_{\fp}[x,f]\otimes_{R_{\fp}}
E_{R_{\fp}}(R_{\fp}/\fp R_{\fp})))^c,$$ and $$(0:_{R_{\fp}}
\Delta_{R_{\fp}}(R_{\fp}[x,f]\otimes_{R_{\fp}}
E_{R_{\fp}}(R_{\fp}/\fp R_{\fp}))) = (0:_R \Delta_R(R[x,f]\otimes_R
E_R(R/\fp)))^e.$$
\end{enumerate}
\end{prop}

\begin{proof} Matsumura \cite[Theorem 18.4]{HM} is one reference for
the well-known facts that multiplication by an $s \in R \setminus
\fp$ provides an automorphism of $E_R(R/\fp)$ (so that $E_R(R/\fp)$
has a natural structure as an $R_{\fp}$-module), and that, as an
$R_{\fp}$-module, $E_R(R/\fp) \cong E_{R_{\fp}}(R_{\fp}/\fp
R_{\fp})$.

(i) Let $n \in \nn$. Since multiplication by an $s \in R \setminus
\fp$ provides an automorphism of $Rx^n \otimes_RE_R(R/\fp)$, the
latter $R$-module has a natural structure as an $R_{\fp}$-module. It
is straightforward to check that $\bigoplus_{n\in\nn} Rx^n \otimes
E_R(R/\fp)$ has a natural structure as a graded left
$R_{\fp}[x,f]$-module, as claimed.

(ii) This is also straightforward, once it has been checked that,
for $n \in \nn$, there is a map
$$
\lambda_n : R_{\fp}[x,f]\otimes_{R_{\fp}}E_R(R/\fp) \lra
R[x,f]\otimes_{R}E_R(R/\fp)
$$
for which $\lambda_n((r/s)x^n\otimes h) =
rs^{p^n-1}x^n\otimes(1/s)h$ for all $h \in E_R(R/\fp)$, $r \in R$
and $s \in R \setminus \fp$.

(iii) The natural image in $R_{\fp}$ of a $c \in R^{\circ}$ lies in
$(R_{\fp})^{\circ}$, and it is immediate from this that
$$
\Delta_R(R[x,f]\otimes_R E_R(R/\fp)) \subseteq
\Delta_{R_{\fp}}(R[x,f]\otimes_R E_R(R/\fp)).
$$
Now let $\zeta \in \Delta_{R_{\fp}}(R[x,f]\otimes_R E_R(R/\fp))$, so
that there exists $c \in R$ such that $c/1 \in (R_{\fp})^{\circ}$
and $(c/1)x^n\zeta = 0$ for all $n \gg 0$. Since
$(R_{\fp}(c/1))^{ce} = R_{\fp}(c/1)$, there exists $c' \in
R^{\circ}$ such that $c'x^n\zeta = 0$ for all $n \gg 0$. Hence
$\zeta \in \Delta_R(R[x,f]\otimes_R E_R(R/\fp))$.

(iv) It follows from part (iii) that
$$(0:_R
\Delta_R(R[x,f]\otimes_R E_R(R/\fp))) = (0:_{R_{\fp}}
\Delta_{R_{\fp}}(R[x,f]\otimes_R E_R(R/\fp)))^c,$$ so that
$$(0:_{R_{\fp}} \Delta_{R_{\fp}}(R[x,f]\otimes_R E_R(R/\fp))) = (0:_R \Delta_R(R[x,f]\otimes_R
E_R(R/\fp)))^e.$$ Since the graded left $R_{\fp}[x,f]$-modules
$R[x,f]\otimes_R E_R(R/\fp)$,
$R_{\fp}[x,f]\otimes_{R_{\fp}}E_R(R/\fp)$ and
$R_{\fp}[x,f]\otimes_{R_{\fp}}E_{R_{\fp}}(R_{\fp}/\fp R_{\fp})$ are
isomorphic, the claims follow.
\end{proof}

In the following corollary, we do not assume that $R$ is local.
Recall that a big test element $c$ for $R$ is said to be a {\em
locally stable big test element for $R$} if and only if, for every
$\fp \in \Spec (R)$, the natural image $c/1$ of $c$ in $R_{\fp}$ is
a big test element for $R_{\fp}$. It is known that a completely
stable big test element for $R$ (that is, a big test element $c$ for
$R$ such that, for every $\fp \in \Spec (R)$, the natural image
$c/1$ of $c$ in $R_{\fp}$ is a big test element for
$\widehat{R_{\fp}}$) is automatically a locally stable big test
element for $R$: see statement (e) on page 64 of Hochster's lecture
course mentioned in the Introduction. In the following corollary, we
prove the stronger statement that {\em every\/} big test element for
$R$ is automatically a locally stable big test element for $R$.

\begin{cor}
\label{pl.7} Every big test element for $R$ is automatically a
locally stable big test element for $R$.
\end{cor}

\begin{proof} Let $c \in R^{\circ}$ be a big test element for $R$.
By Example \ref{jan.5}, this means that, for every $R$-module $M$,
we have $c \in (0:_R\Delta(R[x,f]\otimes_RM))$.  Let $\fp \in
\Spec(R)$. In particular, $c \in
(0:_R\Delta(R[x,f]\otimes_RE_R(R/\fp)))\cap R^{\circ}.$ It now
follows from Proposition \ref{feb.1}(iv) that, in the ring
$R_{\fp}$, we have
$$c/1 \in (0:_{R_{\fp}} \Delta_{R_{\fp}}(R_{\fp}[x,f]\otimes_{R_{\fp}}
E_{R_{\fp}}(R_{\fp}/\fp R_{\fp}))) \cap (R_{\fp})^{\circ}.$$
Therefore $c/1$ is a big test element for $R_{\fp}$, by Theorem
\ref{pl.5c}.
\end{proof}

\begin{lem}
\label{jan.16} Suppose that $(R,\fm)$ is local. Recall that
$E_R(R/\fm)$ has a natural structure as a module over the completion
$(\widehat{R},\widehat{\fm})$ of $R$, and, as an
$\widehat{R}$-module, $E_R(R/\fm) \cong
E_{\widehat{R}}(\widehat{R}/\widehat{\fm})$.

\begin{enumerate}
\item The map $\varphi : R[x,f]\otimes_RE_R(R/\fm) \lra
\widehat{R}[x,f]\otimes_{\widehat{R}}E_R(R/\fm)$ for which
$$\varphi(rx^i \otimes h) = rx^i \otimes h \quad \text{for all~} r \in R,\ i \in
\nn \text{~and~} h \in E_R(R/\fm)$$ is a homogeneous
$R[x,f]$-isomorphism.
\item  If there exists $c \in R^{\circ}$ which
is a big test element for $\widehat{R}$, then
$$
\Delta_{R}(\widehat{R}[x,f]\otimes_{\widehat{R}}E_R(R/\fm)) =
\Delta_{\widehat{R}}(\widehat{R}[x,f]\otimes_{\widehat{R}}E_R(R/\fm)).
$$
\item If there exists $c \in R^{\circ}$ which
is a big test element for $\widehat{R}$, then
$$(0:_R\Delta_{R}(R[x,f]\otimes_RE_R(R/\fm))) = R \cap
(0:_{\widehat{R}}\Delta_{\widehat{R}}(\widehat{R}[x,f]\otimes_{\widehat{R}}
E_{\widehat{R}}(\widehat{R}/\widehat{\fm}))),$$ so that every big
test element for $R$ is automatically a big test element for
$\widehat{R}$.
\end{enumerate}
\end{lem}

\begin{proof} (i) It is straightforward to construct a proof of this based on the fact that each element of
$E_R(R/\fm)$ is annihilated by some power of $\fm$, so that, for $i
\in\nn$, each element of $Rx^i \otimes_RE_R(R/\fm)$ is annihilated
by some power of $\fm$. The details of this are left to the reader.

(ii),(iii) First note that $R^{\circ} \subseteq
\widehat{R}^{\circ}$, and it is automatic from this that
$$
\Delta_{R}(\widehat{R}[x,f]\otimes_{\widehat{R}}E_R(R/\fm))
\subseteq
\Delta_{\widehat{R}}(\widehat{R}[x,f]\otimes_{\widehat{R}}E_R(R/\fm)).
$$
Now suppose that there exists $c \in R^{\circ}$ which is a big test
element for $\widehat{R}$, so that $c \in R^{\circ}\cap
(0:_{\widehat{R}}\Delta_{\widehat{R}}(\widehat{R}[x,f]\otimes_{\widehat{R}}E_R(R/\fm)))$
by Theorem \ref{pl.5c}. Then each element of
$\Delta_{\widehat{R}}(\widehat{R}[x,f]\otimes_{\widehat{R}}E_R(R/\fm))$
is annihilated by $cx^n$ for all $n \in \nn$, and so belongs to
$\Delta_{R}(\widehat{R}[x,f]\otimes_{\widehat{R}}E_R(R/\fm))$.
Therefore $$
\Delta_{R}(\widehat{R}[x,f]\otimes_{\widehat{R}}E_R(R/\fm)) =
\Delta_{\widehat{R}}(\widehat{R}[x,f]\otimes_{\widehat{R}}E_R(R/\fm)),
$$ and part (i) shows that the $R$-annihilator of this is the
$R$-annihilator of $$\Delta_R(R[x,f]\otimes_RE_R(R/\fm)).$$ Since
$E_R(R/\fm) \cong E_{\widehat{R}}(\widehat{R}/\widehat{\fm})$ as
$\widehat{R}$-modules, we deduce that
$$(0:_R\Delta_{R}(R[x,f]\otimes_RE_R(R/\fm))) = R \cap
(0:_{\widehat{R}}\Delta_{\widehat{R}}(\widehat{R}[x,f]\otimes_{\widehat{R}}
E_{\widehat{R}}(\widehat{R}/\widehat{\fm}))).$$ An appeal to Theorem
\ref{pl.5c}, in conjunction with the observation that $R^{\circ}
\subseteq \widehat{R}^{\circ}$, now completes the proof.
\end{proof}

\section{\it A little tight closure theory, but without the Gamma
construction}

One of the aims of this paper is the construction of test elements
without use of the Gamma construction. For explanations of what it
means to say that $R$ is $F$-regular, weakly $F$-regular,
$F$-rational or $F$-pure, see \cite[Definitions 1.2 and 3.11 and
Exercise 2.11]{Hunek96}. All the results listed in the following
reminder can be proved without use of the Gamma construction, and,
when it is not obvious how to do this, some hints are provided.

\begin{rmds}
\label{tcjalg.1} (i) (See Hochster--Huneke \cite[Corollary (5.11),
Proposition (8.7)]{HocHun90}.) \linebreak \indent \quad \quad \; If
$R$ is weakly $F$-regular, then it is normal and $F$-pure.
\begin{enumerate}
\setcounter{enumi}{1}
\item (See Huneke \cite[Proposition 1.5(b)]{Hunek98}.) Let $W$ be a multiplicatively closed subset of
$R$, and let $\fa$ be an ideal of $R$ generated by a regular
sequence. Then $(\fa W^{-1}R)^* = \fa^*W^{-1}R$.
\item (See Hochster--Huneke \cite[Proposition (4.7)]{HocHun94}.)
Assume that $(R,\fm)$ is local and Gorenstein. Then $R$ is weakly $F$-regular if and only
if the ideal generated by one single system of parameters for $R$ is
tightly closed.
\item (See Hochster--Huneke \cite[Proposition (4.7)]{HocHun94}.)
Suppose that $R$ is Gorenstein (but not necessarily local).
Then the following statements are equivalent:
\begin{enumerate}
\item $R$ is $F$-rational;
\item $R$ is weakly $F$-regular;
\item $R$ is $F$-regular.
\end{enumerate}
\item (See K. E. Smith \cite[Theorem
2.6]{Smi97}.) Suppose that $(R,\fm)$ is local; let $\dim R = d$.
Then $H^d_{\fm}(R)$, with its natural left $R[x,f]$-module
structure, is simple over $R[x,f]$ if
\begin{enumerate} \item $R$ is complete, Gorenstein and weakly
$F$-regular, or
\item $R$ is regular.
\end{enumerate}
\end{enumerate}
\end{rmds}

\begin{proof}[Hints for avoidance of the Gamma construction.] (i)
For each finitely generated $R$-module $N$, the map $\psi_N : N \lra
Rx\otimes_RN$ for which $\psi_N(g) = x\otimes g$ for all $g \in N$
has $\Ker \psi_N \subseteq 0^*_N$.

(ii) Huneke's proof in \cite[Proposition 1.5(b)]{Hunek98} works
without the hypothesis that $R$ is excellent and local, and there is
no need to employ a test element in the final paragraph.

(iii),(iv) Once one has part (ii) available, these can be proved
(without use of test elements) via the arguments in Huneke
\cite[Theorem 1.5(1),(2),(3),(5)]{Hunek96}.

(v) The argument used by Smith does not need test elements in the
complete, Gorenstein, weakly $F$-regular case (a); the regular case
(b) can then be dealt with by passage to the completion.
\end{proof}

The next result is a consequence of \ref{tcjalg.1}(i),(iii),(iv) and
it will help to simplify some of the proofs later in the paper.

\begin{prop}
\label{tc.postant} Suppose that $R$ satisfies condition $(R_0)$ and
that its distinct minimal prime ideals are $\fp_1, \ldots, \fp_h$;
set $\fn := \sqrt{0} = \fp_1 \cap \cdots \cap \fp_h$. Let $c \in
R^{\circ}$ be such that $R_c$ is $F$-regular. Let $\overline{c}$
denote the natural image of $c$ in $R/\fn$. Let $j \in
\{1,\ldots,h\}$ and let $\widetilde{c}$ denote the natural image of
$c$ in $R/\fp_j$.

\begin{enumerate}
\item There is a ring isomorphism $R_c \cong (R/\fn)_{\overline{c}}$; also, $(R/\fp_j)_{\widetilde{c}}$
is isomorphic as a ring to a direct factor of $R_c$, and therefore
to a ring of fractions of $R_c$.

\item If some power of $\overline{c}$ is a big test element for
$R/\fn$, then some power of $c$ is a big test element for $R$.

\item If, for each $i \in \{1,\ldots,h\}$, some power of the natural
image of $c$ in $R/\fp_i$ is a big test element for $R/\fp_i$, then
some power of $\overline{c}$ is a big test element for $R/\fn$, so
that, by part\/ {\rm (ii)}, some power of $c$ is a big test element
for $R$.
\end{enumerate}
\end{prop}

\begin{proof} (i) For each
$\fp \in \Spec(R)$ such that $c \not\in \fp$, the localization
$(R_c)_{\fp R_c}$ is weakly $F$-regular, and so is an integral
domain (by \ref{tcjalg.1}(i)). It is elementary to deduce from this
that there exists $w \in \N$ such that $c^w\fn = 0$, and that the
minimal prime ideals of $R_c$ are pairwise comaximal. The claims now
all follow easily.

(ii) Let $w$ be as in the proof of part (i). Suppose that $v \in
\nn$ is such that $\overline{c}^v$ is a big test element for
$R/\fn$. Let $M$ be an $R$-module and let $m \in 0^*_M$. Let $n \in
\nn$.  It is straightforward to show that $c^vx^n\otimes m$ can be
expressed as $\sum_{i=1}^t b_ix^n \otimes m_i$ for some $b_1,
\ldots, b_t \in \fn$ and $m_1, \ldots, m_t \in M$. Since $c^w\fn =
0$, we must have $c^{w+v}x^n\otimes m = 0$. It follows that
$c^{w+v}$ is a big test element for $R$.

(iii) (The ideas used in this proof come from Hochster's and
Huneke's proof of \cite[Proposition (6.25)]{HocHun90}.) By parts (i)
and (ii), we can assume that $R$ is reduced. The natural ring
homomorphism $ R_c \lra \prod_{i=1}^h R_c/\fp_i R_c$ is an
isomorphism. It follows that the cokernel $K$ of the natural
$R$-monomorphism $\psi : R \lra \bigoplus_{i=1}^h R/\fp_i$ satisfies
$K_c = 0$; since $K$ is finitely generated, there exists $u \in \nn$
such that $c^uK = 0$.

Let $j \in \{1,\ldots,h\}$, and consider the element of
$\bigoplus_{i=1}^h R/\fp_i$ whose $j$th component is $1 + \fp_j$ and
all of whose other components are zero. Since $c^uK = 0$, there
exists $d_j \in R$ such that $ d_j \in \bigcap_{i=1,i\neq j}^h\fp_i$
and $d_j - c^u \in \fp_j$. Note that $d_j\fp_j \subseteq
\bigcap_{i=1}^h\fp_i = \fn = 0$.

Suppose now that $y \in \nn$ is such that, for all $i = 1, \ldots,
h$, the $y$th power of the natural image of $c$ in $R/\fp_i$ is a
big test element for $R/\fp_i$. Let $M$ be an $R$-module and let $m
\in 0^*_M$. It is straightforward to show that $d_jc^yx^n\otimes m =
0$ in $R[x,f]\otimes_RM$, for all $n \in\nn$. This is true for each
$j \in \{1,\ldots,h\}$. Therefore
$\left(\sum_{i=1}^hd_i\right)c^yx^n\otimes m = 0$ in
$R[x,f]\otimes_RM$, for all $n \in\nn$. But $\psi(\sum_{i=1}^hd_i) =
\psi(c^u)$, so that $\sum_{i=1}^hd_i = c^u$ because $\psi$ is a
monomorphism. Therefore $c^{u+y}x^n\otimes m = 0$ in
$R[x,f]\otimes_RM$, for all $n \in\nn$. Hence $c^{u+y}$ is a big
test element for $R$.
\end{proof}

The significance of Proposition \ref{tc.postant} is as follows: if
$R$ satisfies condition $(R_0)$ and $c \in R^{\circ}$ is such that
$R_c$ is Gorenstein and weakly $F$-regular (so that $R_c$ is
$F$-regular by \ref{tcjalg.1}(iv)), then, in order to prove that
some power of $c$ is a big test element for $R$, it is enough to do
so under the additional assumption that $R$ is a domain. In the case
where $R$ is local, it is very helpful if we can reduce to the case
where $R$ is not only a local domain, but also complete, and the
following proposition will enable us to do this in some
circumstances.

\begin{prop}\label{jul.4} Suppose that $R \subseteq R'$ is a faithfully flat
extension of (commutative Noetherian) rings of characteristic $p$.
Suppose that $c \in R^{\circ}$ is a big test element for $R'$. Then
$c$ is a big test element for $R$.
\end{prop}

\begin{proof} Observe that $R^{\circ}
\subseteq R'^{\circ}$. This result is an easy consequence of the
faithful flatness and the fact that, for each $R$-module $M$, there
is a $\Z$-isomorphism
$$
\lambda : R'\otimes_R(R[x,f] \otimes_RM) \stackrel{\cong}{\lra}
R'[x,f] \otimes_{R'}(R'\otimes_RM)
$$
for which $\lambda(r'\otimes(rx^n\otimes m)) =
r'rx^n\otimes(1\otimes m)$ for all $n \in \nn$, $r' \in R'$, $r\in
R$ and $m\in M$.
\end{proof}

\section{\it Homomorphic images of regular rings of characteristic $p$}
\label{hrr}

In subsequent sections, we are going to study the situation where
$R$ is a homomorphic image of a regular ring of characteristic $p$.
One motivation for this approach comes from I. S. Cohen's Structure
Theorem for complete local rings, which shows that a complete local
ring of characteristic $p$ is a homomorphic image of a complete
regular local ring of characteristic $p$. However, the ideas in this
section will also be applied in situations where $R$ is not local,
and in situations where $R$ is local but not complete.

\begin{ntn}
\label{jun.3} We shall denote $1 + p + p^2 + \cdots +p^{n-1}$, where
$n \in \N$, by $\omega_n$, and we shall interpret $\omega_0$ as $0$.
Note that
\begin{enumerate}
\item $p\omega_n < 1 + p\omega_n = \omega_{n+1}$,
\item $(1-p)\omega_n = 1 - p^n$, and
\item $\omega_1 = 1$.
\end{enumerate}
\end{ntn}

\begin{lem}
\label{jun.2} Let $S$ be a regular ring of characteristic $p$, let
$\fp$ be a prime ideal of $S$, let $n \in \N$ and let $u \in
(\fp^{[p]}:\fp) \setminus \fp^{[p]}$. Then
$\left(\fp^{[p^n]}:u^{\omega_n}\right) = \fp$ for all $n \in \nn$.
\end{lem}

\begin{proof} This was proved in the case where $S$ is a regular local ring in
\cite[Proposition 2.9]{Fpurhastest}, and the same proof works in
this more general situation.
\end{proof}

We shall use the construction in the next theorem more than once.

\begin{prop}\label{jun.1} Suppose that $R$ is an integral domain and a homomorphic image
$S/\fa$ of a regular ring $S$ of characteristic $p$, where $\fa$ is
a non-zero prime ideal of $S$.

Let $\fq \in \Spec(S)$ be such that $\fq \supseteq \fa$ and
$R_{\fq/\fa}$ is $F$-pure. Then there exist $u_1, \ldots, u_l \in
(\fa^{[p]} : \fa) \setminus \fq^{[p]}$ such that $(\fa^{[p]} : \fa)
= Su_1 + \cdots + Su_l + \fa^{[p]}$.
\end{prop}

\begin{proof} Suppose, inductively, that $j \in \nn$ and
$u_1, \ldots, u_j \in (\fa^{[p]}:\fa) \setminus \fq^{[p]}$ have been
constructed so that $(\sum_{i=1}^kSu_i)_{k=0}^j$ form a strictly
ascending chain of ideals. This is certainly true when $j = 0$.

If the natural images of $u_1, \ldots, u_j$ in $T :=
(\fa^{[p]}:\fa)/\fa^{[p]}$ generate $T$, then we have achieved our
aim. We therefore suppose that this is not the case, so that
$$\fa^{[p]} + Su_1 + \cdots + Su_j \subset (\fa^{[p]}:
\fa).$$ (The symbol `$\subset$' is reserved to denote strict
inclusion.)

Since $R_{\fq/\fa} \cong S_{\fq}/\fa S_{\fq}$ is $F$-pure, it
follows from Fedder's Theorem \cite[Theorem 1.12]{Fedde83} that $
((\fa S_{\fq})^{[p]} : \fa S_{\fq}) \not\subseteq (\fq
S_{\fq})^{[p]},$ so that $(\fa^{[p]}: \fa) \not\subseteq \fq^{[p]}$.
Therefore, by ideal avoidance (in the form presented in
\cite[Theorem 81]{IK}), there exists
$$
u_{j+1} \in (\fa^{[p]}:\fa) \setminus \fq^{[p]} \cup \left(
\fa^{[p]} + Su_1 + \cdots + Su_j\right).
$$
Note that $Su_1 + \cdots + Su_j \subset Su_1 + \cdots + Su_j +
Su_{j+1}$. This completes the inductive step.

Since $R$ is Noetherian, this inductive argument yields $u_1,
\ldots, u_l \in (\fa^{[p]}:\fa) \setminus \fq^{[p]}$ whose natural
images in $T$ generate $T$.
\end{proof}

\section{\it Homomorphic images of regular local rings of characteristic $p$}
\label{hrlr} If the local ring $(R,\fm)$ has a test element for
ideals, then $R$ must satisfy condition $(R_0)$, by Lemma
\ref{in.3}. We have seen in Corollary \ref{pl.6} that there exists a
big test element for $R$ if we can extend the $R$-module structure
on $E := E_R(R/\fm)$ to a structure as a left $R[x,f]$-module in a
sufficiently non-trivial way so that $(0:_R\Gamma_x(E))$ has
positive height. Now $E$ has a natural structure as a module over
the completion $\widehat{R}$ of $R$, and a structure as a left
$R[x,f]$-module on $E$ induces, in a unique way, a structure as a
left $\widehat{R}[x,f]$-module on it extending the $R[x,f]$-module
structure. As an $\widehat{R}$-module, $E \cong
E_{\widehat{R}}(\widehat{R}/\fm \widehat{R})$, and I. S. Cohen's
Structure Theorem for complete local rings ensures that
$\widehat{R}$ is a homomorphic image of a complete regular local
ring of characteristic $p$.

With these considerations in mind, we are going, in this section, to
consider possible left $R[x,f]$-module structures on $E$ when $R$ is
a homomorphic image of a regular local ring of characteristic $p$.
Some preparatory results were provided in \S \ref{hrr}. We now
recall some other theory relevant to this situation that was
developed in \cite[\S 2]{Fpurhastest}.

\begin{disc} \label{hrlr.1} Throughout this section, we shall assume
that $R = S/\fa$, where $S$ is a regular local ring of
characteristic $p$ and $\fa$ is a proper ideal of $S$.

We shall denote the set of prime ideals of $S$ that contain $\fa$ by
$\Var(\fa)$. For $s \in S$, we shall denote the natural image of $s$
in $R$ by $\overline{s}$. We shall only assume that $S$ is complete
when this is explicitly stated. We shall use $\fn$ to denote the
maximal ideal of $S$, so that $\fm := \fn/\fa$ is the maximal ideal
of $R$. Set $E := E_S(S/\fn)$, and note that $(0:_E\fa) \cong
E_R(R/\fm)$ as $R$-modules. We shall interpret $(0:_E\fa)$ as
$E_R(R/\fm)$.

Observe that $R[x,f]$ is a homomorphic image of $S[x,f]$ under a
homomorphism that extends the natural surjective homomorphism from
$S$ to $R$ and maps the indeterminate $x$ to $x$. Let $y$ be a new
indeterminate. Since $E \cong H^{\dim S}_{\fn}(S)$, the $S$-module
$E$ has a natural structure as a left $S[y,f]$-module. For any
element $u \in S$, it is easy to see that we can endow $E$ with a
structure of left $S[x,f]$-module under which $xe = uye$ for all $e
\in E$. A recurring theme of this section is the idea of trying to
choose a $u$ as above in such a way that $(0:_E\fa) = E_R(R/\fm)$ is
an $S[x,f]$-submodule of $E$, and so becomes a left $R[x,f]$-module.

We now remind the reader of some results from \cite[\S
2]{Fpurhastest} about this situation.
\end{disc}

\begin{rmds}
\label{jun.5} Let the situation and notation be as in \ref{hrlr.1}.
\begin{enumerate}
\item Let $u
\in S$. Put the left $S[x,f]$-module structure on $E$ for which $xe
= uye$ for all $e \in E$. Then $(0:_E\fa)$ is an $S[x,f]$-submodule
of $E$ if and only if $u \in (\fa^{[p]} : \fa)$. See \cite[Lemma
2.4]{Fpurhastest}.

When this condition is satisfied, $E_R(R/\fm)$ is a left
$R[x,f]$-module with $xe = uye$ for all $e \in (0:_E\fa) =
E_R(R/\fm)$ and $re = se$ for any $r \in R$ and $s \in S$ for which
$\overline{s} = r$. The words `use $u$ to furnish $E_R(R/\fm)$ with
a structure as a left $R[x,f]$-module' will always refer to this
construction.
\item Suppose
that $S$ is complete, and that $(0:_E\fa)$ has a left
$R[x,f]$-module structure that extends its $R$-module structure.
Then there exists $u \in (\fa^{[p]} : \fa)$ such that $xe = uye$ for
all $e \in (0:_E\fa)$. See M. Blickle \cite[Proposition
3.36]{Blickle} (or \cite[Lemma 2.5]{Fpurhastest} for a short,
self-contained proof).
\item If the local ring $R = S/\fa$ is Gorenstein, then
$(\fa^{[p]}:\fa)/\fa^{[p]}$ is a cyclic $S$-module. See
\cite[Corollary 2.11]{Fpurhastest}.
\item Let $\fp
\in \Spec (S)$ and $u \in (\fp^{[p]} : \fp)$. Bearing in mind part
(i) above, put the left $S[x,f]$-module structure on $E$ for which
$xe = uye$ for all $e \in E$, so that $(0:_E\fp)$ is an
$S[x,f]$-submodule of $E$. If $u \not\in \fp^{[p]}$, then
$(0:_E\fp)$ is not $x$-torsion. See \cite[Proposition
2.9]{Fpurhastest}.
\end{enumerate}
\end{rmds}

\begin{lem}\label{fp.210} Suppose that $R$ is excellent and
satisfies the condition $(R_0)$. Let $\fp \in \Spec (R)$ be such
that $R_{\fp}$ is regular. Then there exists $c \in
R^{\circ}\setminus \fp$ such that $R_c$ is regular.
\end{lem}

\begin{proof} Let $\Min(R)$ denote the set of minimal prime ideals
of $R$. By hypothesis, $R_{\fq}$ is a field for each $\fq \in
\Min(R)$. Since $R$ is excellent, there is a radical ideal $\fd$ of
$R$ such that
$
\{ \fq \in \Spec (R) : R_{\fq} \text{~is not regular}\} =  \{ \fq
\in \Spec (R) : \fq \supseteq \fd\}.
$
By prime avoidance, there exists ${\textstyle c \in \fd \setminus
\fp \bigcup \left(\bigcup_{\fq \in \Min(R)}\fq\right)},$ and $R_c$
is regular.
\end{proof}

\begin{lem}
\label{feb.2} Suppose that $(R,\fm)$ is local and excellent.
\begin{enumerate}
\item If $R$ satisfies condition $(R_0)$, then so also does $\widehat{R}$.
\item If $c \in R$ is such that $R_c$ is regular, then
$\widehat{R}_c$ is regular.
\item If $c \in R$ is such that $R_c$ is Gorenstein, then
$\widehat{R}_c$ is Gorenstein.
\end{enumerate}
\end{lem}

\begin{proof} These claims are all easy consequences of the fact
that all the `formal' fibre rings of the flat local inclusion
homomorphism $R \lra \widehat{R}$ are regular. Note that, for $c \in
R$, the fibre rings of the flat ring homomorphism $R_c \lra
\widehat{R}_c$ induced by inclusion are rings of fractions of the
formal fibres of $R$, and so are regular.
\end{proof}

\begin{thm}
\label{jan.17} Suppose that $(R,\fm)$ is excellent and local, and
that $R$ satisfies condition $(R_0)$. Then, for each $c \in
R^{\circ}$ for which $R_c$ is regular, some power of $c$ is a big
test element for $R$.
\end{thm}

\begin{proof} By Lemma \ref{feb.2}, the completion $(\widehat{R},
\widehat{\fm})$ of $R$ satisfies condition $(R_0)$, and
$\widehat{R}_c$ is regular. Note also that $c \in
\widehat{R}^{\circ}$. Complete local rings are excellent. It
therefore follows from Propositions \ref{tc.postant} and \ref{jul.4}
that we can assume that $R$ is a complete local domain.

We appeal to I. S. Cohen's Structure Theorem for complete local
rings containing a subfield and write $R = S/\fa$, where $S$ is a
complete regular local ring of characteristic $p$ and $\fa$ is a
non-zero prime ideal of $S$.

Let $\fp \in \Spec(R)$ be such that $R_{\fp}$ is regular. Let $\fq
\in \Var(\fa)$ be such that $\fq/\fa = \fp$. Set $T :=
(\fa^{[p]}:\fa)/\fa^{[p]}$. Apply Proposition \ref{jun.1} to find
$u_1, \ldots, u_l \in (\fa^{[p]}:\fa)\setminus \fq^{[p]}$ whose
natural images in $T$ generate $T$.

Now $ \left((\fa S_{\fq})^{[p]} : \fa S_{\fq}\right) = (\fa
S_{\fq})^{[p]} + (u_1/1)S_{\fq} + \cdots + (u_{l}/1)S_{\fq}.$ Since
$S_{\fq}/\fa S_{\fq} \cong R_{\fp}$ is a Gorenstein local ring, it
follows from \ref{jun.5}(iii) that $\left((\fa S_{\fq})^{[p]} : \fa
S_{\fq}\right)/(\fa S_{\fq})^{[p]}$ is cyclic over $S_{\fq}$;
therefore, there exists $i \in \{1,\ldots,l\}$ such that
$$
\left((\fa S_{\fq})^{[p]} : \fa S_{\fq}\right) =  (\fa
S_{\fq})^{[p]} + (u_i/1)S_{\fq}.
$$
Use $u_i$ to furnish $E_R(R/\fm) = (0:_E\fa)$ with a structure as a
left $R[x,f]$-module. Since $u_i \not\in \fa^{[p]}$, it follows from
\ref{jun.5}(iv) that $(0:_E\fa)$ is not $x$-torsion. Hence
$\Gamma_x(0:_E\fa) \neq (0:_E\fa)$ and $\height
(0:_R\Gamma_x(0:_E\fa)) > 0$ by Matlis duality (see \cite[p.\
154]{SV}, for example). Now use Corollary \ref{pl.6} to see that the
ideal $\fd_0 := (0:_R\Delta(E_R(R/\fm)))$ has positive height and
all its members in $R^{\circ}$ are big test elements for $R$.

Our next aim is to show that $\fd_0 \not\subseteq \fp$. Recall from
above that $u_i$ was chosen so that
$$
\left((\fa S_{\fq})^{[p]} : \fa S_{\fq}\right) =  (\fa
S_{\fq})^{[p]} + (u_i/1)S_{\fq}.
$$
Let $G$ denote the injective envelope, over $R_{\fp}$, of the simple
$R_{\fp}$-module, and use $u_i/1$ in conjunction with the
isomorphism $R_{\fp} \cong S_{\fq}/\fa S_{\fq}$ to put a left
$R_{\fp}[x,f]$-module structure on $G$. Let $\fk_0$ be the ideal of
$S$ that contains $\fa$ and is such that $\fd_0 = \fk_0/\fa$. Since
$R$ is complete, it follows from Matlis duality that
$$
\Delta(E_R(R/\fm)) = (0:_{E_R(R/\fm)}\fd_0) = (0:_E\fk_0),
$$
and since $\Delta(E_R(R/\fm))$ is an $R[x,f]$-submodule of
$E_R(R/\fm)$, we deduce from \ref{jun.5}(i) that $u_i \in
\left((\fk_0)^{[p]}:\fk_0\right)$. Consequently, in $S_{\fq}$, we
have $ u_i/1 \in \left((\fk_0S_{\fq})^{[p]}:\fk_0S_{\fq}\right), $
and so $(0:_G \fk_0S_{\fq}/\fa S_{\fq}) = (0:_G\fd_0R_{\fp})$ is an
$R_{\fp}[x,f]$-submodule of $G$, again by \ref{jun.5}(i). Now $G$ is
simple as an $R_{\fp}[x,f]$-module, by \cite[Corollary
2.13]{Fpurhastest}. (Note that the appeal to \cite[Theorem
2.6]{Smi97} at the end of the proof of that corollary can be
replaced by use of \ref{tcjalg.1}(v)(b), which can be proved without
use of test elements.) Therefore, since a module over a (not
necessarily complete) local ring has the same annihilator as its
Matlis dual, $\fd_0R_{\fp}$ must be $0$ or $R_{\fp}$. Since $\height
\fd_0R_{\fp} \geq 1$, we must have $\fd_0R_{\fp} = R_{\fp}$.
Therefore $\fd_0 \not\subseteq \fp$.

Since $R$ has a big test element, $\height
(0:_R\Delta(R[x,f]\otimes_RE_R(R/\fm))) > 0$, and also $\fd_0
\subseteq (0:_R\Delta(R[x,f]\otimes_RE_R(R/\fm)))$, by Theorem
\ref{pl.5c}. Therefore
$$(0:_R\Delta(R[x,f]\otimes_RE_R(R/\fm))) \not\subseteq \fp.$$
We have therefore shown that
$(0:_R\Delta(R[x,f]\otimes_RE_R(R/\fm)))$ is not contained in any
$\fp' \in \Spec (R)$ for which $R_{\fp'}$ is regular. It follows
that, if $c \in R^{\circ} = R \setminus \{0\}$ is such that $R_c$ is
regular, then $c \in
\sqrt{(0:_R\Delta(R[x,f]\otimes_RE_R(R/\fm)))}$, so that some power
of $c$ is a big test element for $R$.
\end{proof}

\begin{cor}
\label{jan.18} Suppose that $(R,\fm)$ is excellent and local, and
that $R$ satisfies condition $(R_0)$. Then every big test element
for $R$ is automatically a big test element for $\widehat{R}$.
\end{cor}

\begin{proof} Let $c \in R^{\circ}$ be such that $R_c$ is regular. We deduce from
Lemma \ref{feb.2} that $\widehat{R}$ satisfies $(R_0)$ and that
$\widehat{R}_c$ is regular. By Theorem \ref{jan.17}, there exists $m
\in \N$ such that $c^m$ is a big test element for both $R$ and
$\widehat{R}$. Therefore, by Lemma \ref{jan.16}(iii), every big test
element for $R$ is automatically a big test element for
$\widehat{R}$.
\end{proof}

\begin{cor}
\label{jan.18a} Suppose that $R$ is excellent (but not necessarily
local). Then a big test element for $R$ is automatically a locally
stable big test element for $R$ and a completely stable big test
element for $R$.
\end{cor}

\begin{proof} Let $c \in R^{\circ}$ be a big test element for $R$.
Then $R$ must satisfy condition $(R_0)$, by Lemma \ref{in.3}. Let
$\fp \in \Spec (R)$. By Corollary \ref{pl.7}, the natural image
$c/1$ of $c$ in $R_{\fp}$ is a big test element for $R_{\fp}$.

Now $R_{\fp}$ is excellent and local, and it satisfies condition
$(R_0)$; therefore $c/1$ is a big test element for
$\widehat{R_{\fp}}$, by Corollary \ref{jan.18}.
\end{proof}

\section{\it Some refinements}\label{wys}
We have proved in \S \ref{hrlr} that if $(R,\fm)$ is excellent and
local, and satisfies condition $(R_0)$, then each $c \in R^{\circ}$
for which $R_c$ is regular has a power that is a (locally stable and
completely stable) big test element for $R$. Moreover, we have shown
how this can be achieved without use of the Gamma construction. We
would now like to strengthen this result by replacing `$R_c$ is
regular' by `$R_c$ is Gorenstein and weakly $F$-regular', and we
would like the arguments that achieve this to be independent of the
Gamma construction.

It would be very helpful to have available the result that, if
$(R,\fm)$ is excellent, local, Gorenstein and weakly $F$-regular,
then the completion $\widehat{R}$ is also weakly $F$-regular. This
is indeed a result (of Hochster and Huneke \cite[Theorem
(7.24)]{HocHun94}), but their proof is dependent on the Gamma
construction. Consequently, the following less ambitious result is
provided, but with a proof that is independent of the Gamma
construction.

\begin{prop}\label{jun.101} {\rm (Compare Hochster--Huneke \cite[Theorem
(7.24)]{HocHun94} and I. M. Aberbach \cite[Theorem 3.4]{Aberb01}.)}
Suppose that $(R,\fm)$ is excellent and local, and that $(R',\fM)$
is another excellent local ring of characteristic $p$ that is a
faithfully flat extension of $R$; assume that, for all $\fp \in
\Spec(R)$, the fibre ring of the inclusion homomorphism $R \lra R'$
over $\fp$ is regular. Then $R$ is Gorenstein and weakly $F$-regular
if and only if $R'$ is Gorenstein and weakly $F$-regular.
\end{prop}

\begin{proof}[Proof (without use of the Gamma construction).]
Since a weakly $F$-regular local ring of charactersitic $p$ is an
integral domain (see \ref{tcjalg.1}(i)), we can assume that both $R$
and $R'$ are Gorenstein and reduced. As $R$ is excellent, there
exists $c \in R^{\circ}$ such that $R_c$ is regular. The hypothesis
about the fibre rings ensures that $R'_c$ is regular. Since $c \in
R^{\circ}\subseteq R'^{\circ}$, we can now use Theorem \ref{jan.17}
to deduce that there is a power of $c$ that is a big test element
for both $R$ and $R'$.

Let $x_1, \ldots, x_n$ be a system of parameters for $R$, and let
$z_1, \ldots, z_d \in R'$ be such that their natural images in the
regular local ring $R'/\fm R'$ form a regular system of parameters
for that ring; note that $x_1, \ldots, x_n,z_1, \ldots, z_d$ is a
system of parameters for $R'$. It follows from \ref{tcjalg.1}(iii)
that $R$ is weakly $F$-regular if and only if the ideal of $R$
generated by $x_1, \ldots, x_n$ is tightly closed, and that $R'$ is
weakly $F$-regular if and only if, for some $t \in \N$, the ideal of
$R'$ generated by $x_1, \ldots, x_n,z_1^t, \ldots, z_d^t$ is tightly
closed. We can now appeal to I. M. Aberbach \cite[Proposition
3.2(1)]{Aberb01} to complete the proof.
\end{proof}

\begin{cor}
\label{sr.2} Suppose that $R \subseteq R'$ is a faithfully flat
extension of excellent rings of characteristic $p$ such that all the
fibre rings of the inclusion ring homomorphism are regular, and that
$c \in R$ is such that $R_c$ is Gorenstein and weakly $F$-regular.
Then $R'_c$ is Gorenstein and weakly $F$-regular.
\end{cor}

\begin{proof} The fibre rings of the flat ring homomorphism
$R_c \lra R'_c$ induced by inclusion are rings of fractions of the
fibre rings of the inclusion homomorphism $R \lra R'$, and so are
regular. Therefore $R'_c$ is Gorenstein. Note that, by
\ref{tcjalg.1}(iv), a Gorenstein ring is weakly $F$-regular if and
only if it is $F$-regular.

By \cite[Theorem 1.5(3)]{Hunek96}, in order to prove that $R'_c$ is
weakly $F$-regular, it is enough to prove that every localization
(at a prime ideal) of $R'_c$ is weakly $F$-regular. Therefore, in
order to prove the corollary, it is enough to prove that, for each
$\fP \in \Spec (R')$ such that $c \not\in \fP$, the localization
$R'_{\fP}$ is Gorenstein and weakly $F$-regular.

Let $\fp := \fP \cap R$; note that $c \not\in \fp$, so that
$R_{\fp}$ is Gorenstein and weakly $F$-regular. The inclusion
homomorphism $R \lra R'$ induces a flat local homomorphism $R_{\fp}
\lra R'_{\fP}$ all of whose fibre rings are regular (because they
are rings of fractions of fibre rings of $R \lra R'$). Since
$R_{\fp}$ and $R'_{\fP}$ are excellent, Proposition \ref{jun.101}
therefore shows that $R'_{\fP}$ is Gorenstein and weakly
$F$-regular.
\end{proof}

The proposition below is a modification of \cite[Corollary
2.13]{Fpurhastest}.

\begin{prop}
\label{hrlr.12} Use the notation of\/ {\rm \ref{hrlr.1}}, so that
$(R,\fm)$ is local and equal to the homomorphic image $S/\fa$ of the
regular local ring $S$ of characteristic $p$.

Suppose that $u \in S$ is such that $ (\fa^{[p]}:\fa) = \fa^{[p]} +
Su.$ Use $u$ to furnish $E_R(R/\fm)$ with a structure as a left
$R[x,f]$-module, and denote this left $R[x,f]$-module by $D_1$. If
$R$ is excellent, Gorenstein and weakly $F$-regular, then $D_1$ is a
simple $R[x,f]$-module.
\end{prop}

\begin{proof} Note that, by Proposition \ref{jun.101}, the
completion $\widehat{R}$ is weakly $F$-regular (and Gorenstein).
With this observation, this proposition can be proved by a
straightforward modification of the proof of \cite[Corollary
2.13]{Fpurhastest}: at the end, use the weak version of Smith's
Theorem cited as \ref{tcjalg.1}(v)(a), which can be proved without
use of test elements.
\end{proof}

We are now in a position to prove a strengthening of Theorem
\ref{jan.17}.

\begin{thm}
\label{jan.17gf} Suppose that $(R,\fm)$ is excellent and local, and
that $R$ satisfies condition $(R_0)$. Then, for each $c \in
R^{\circ}$ for which $R_c$ is Gorenstein and weakly $F$-regular,
some power of $c$ is a big test element for $R$.
\end{thm}

\begin{proof} The argument used earlier to prove Theorem
\ref{jan.17} can now be adapted to prove this stronger statement: we
have $c \in \widehat{R}^{\circ}$; it follows from Corollary
\ref{sr.2} that $\widehat{R}_c$ is Gorenstein and weakly
$F$-regular; we can assume that $R$ is a complete local domain; we
should choose $\fp \in \Spec (R)$ such that $R_{\fp}$ is Gorenstein
and weakly $F$-regular (and so $F$-pure by \ref{tcjalg.1}(i)); we
can still use Proposition \ref{jun.1}; and the appeal to
\cite[Corollary 2.13]{Fpurhastest} (in the proof of \ref{jan.17})
should be replaced by an appeal to Proposition \ref{hrlr.12}.
\end{proof}

The final two results in this section present the promised
strengthening of the author's result in \cite[Theorem
4.16]{Fpurhastest} about $F$-pure excellent rings.

\begin{cor}
\label{may.1} Suppose that $R$ is excellent, local and $F$-pure, and
that $c \in R^{\circ}$ is such that $R_c$ is Gorenstein and weakly
$F$-regular. Then $c$ itself is a big test element for $R$.
\end{cor}

\begin{proof} Since $R$ is $F$-pure, it is reduced, and so satisfies
the condition $(R_0)$; therefore, by Theorem \ref{jan.17gf}, some
power of $c$ is a big test element for $R$. Since $R$ is $F$-pure,
the left $R[x,f]$-module $R[x,f]\otimes_RM$ is $x$-torsion-free, for
each $R$-module $M$. In particular, $R[x,f]\otimes_RE_R(R/\fm)$ is
$x$-torsion-free, so that $(0:_R\Gamma_x(R[x,f]\otimes_RE_R(R/\fm)))
= R$. It therefore follows from Corollary \ref{jul.26} that $c$
itself is a big test element for $R$.
\end{proof}

\begin{cor}
\label{may.2} Suppose that $R$ is excellent and $F$-pure (but not
necessary local), and that $c \in R^{\circ}$ is such that $R_c$ is
Gorenstein and weakly $F$-regular. Then $c$ itself is a big test
element for $R$.
\end{cor}

\begin{proof} Let
$\fm \in \Max(R)$. By Hochster and J. L. Roberts \cite[Lemma
6.2]{HocRob74}, the localization $R_{\fm}$ is $F$-pure; it is also
excellent. Since $\big(R_{\fm}\big)_{c/1}$, the ring of fractions of
$R_{\fm}$ with respect to the set of powers of the element $c/1$ of
$(R_{\fm})^{\circ}$, is a ring of fractions of $R_c$, it is
Gorenstein and weakly $F$-regular. (We have used \ref{tcjalg.1}(iv)
here.) It therefore follows from Corollary \ref{may.1} that $c/1$ is
a big test element for $R_{\fm}$. Therefore $ c/1 \in (0:_{R_{\fm}}
\Delta_{R_{\fm}}(R_{\fm}[x,f]\otimes_{R_{\fm}}
E_{R_{\fm}}(R_{\fm}/\fm R_{\fm})))$ in $R_{\fm}$, by Theorem
\ref{pl.5c}. Proposition \ref{feb.1}(iv) now shows that $c \in (0:_R
\Delta_R(R[x,f]\otimes_R E_R(R/\fm)))$. This is true for all $\fm
\in \Max(R)$. Now it is easy to see that there is a homogeneous
isomorphism of graded left $R[x,f]$-modules
$$
R[x,f] \otimes_R \left({\textstyle \bigoplus_{\fm
\in\Max(R)}E_R(R/\fm)} \right) \cong {\textstyle \bigoplus_{\fm
\in\Max(R)} \left(R[x,f]\otimes_RE_R(R/\fm)\right)},
$$
and to deduce from this that
$
c \in \left(0:_R \Delta_R\!\left(R[x,f] \otimes_R\!
\left({\textstyle \bigoplus_{\fm \in\Max(R)}E_R(R/\fm)}
\right)\right)\right).
$
Since $\bigoplus_{\fm \in\Max(R)}E_R(R/\fm)$ is an injective
cogenerator of $R$, another use of Theorem \ref{pl.5c} shows that
$c$ is a big test element for $R$.
\end{proof}

\section{\it Frobenius-intersection-flatness}
\label{if}

\begin{defs}\label{if.1} Recall from \cite[p.\ 41]{HocHun94} that a
commutative $R$-algebra $R'$ is said to be {\em intersection-flat\/}
if $R'$ is flat as an $R$-algebra and, for every non-empty family of
submodules $(M_{\lambda})_{\lambda\in\Lambda}$ of a finitely
generated $R$-module $M$, the natural $R'$-monomorphism
$$
R'\otimes_R{\textstyle
\left(\bigcap_{\lambda\in\Lambda}M_{\lambda}\right)} \lra
{\textstyle \bigcap_{\lambda\in\Lambda}(R'\otimes_RM_{\lambda})}
$$
is an isomorphism. It is easy to check that, if $R'$ is an
intersection-flat $R$-algebra and $R^{\prime\prime}$ is an
intersection-flat $R'$-algebra, then $R^{\prime\prime}$ is an
intersection-flat $R$-algebra.

Let $S$ be a regular ring of characteristic $p$. We shall say that
$S$ is {\em Frobenius-intersection-flat\/}, or {\em $F$-$\cap$-flat
\/} for short, if the Frobenius homomorphism $f : S \lra S$ is
intersection-flat; then, for all $n \in \N$, the $n$th interate $f^n
: S \lra S$ is also intersection-flat.
\end{defs}

An $F$-finite regular ring $S$ of characteristic $p$ is
$F$-$\cap$-flat, because, when viewed as an $S$-module via $f$, it
is finitely generated and flat, and therefore projective. We now
review some more examples, provided by M. Katzman in \cite[\S
5]{Katzm08}, of regular rings of characteristic $p$ that are
$F$-$\cap$-flat.

\begin{exs}\label{if.2} Let $\K$ be a field of characteristic $p$ and let $X_1,\ldots,X_n,Y_1,
\ldots,Y_m$ be independent indeterminates.
\begin{enumerate}
\item The polynomial ring $\K[X_1,\ldots,X_n]$, and each
localization of $\K[X_1,\ldots,X_n]$, are $F$-$\cap$-flat. See
\cite[p.\ 941]{Katzm08}.
\item The ring of formal power series $\K[[X_1,\ldots,X_n]]$ is
$F$-$\cap$-flat (by \cite[Proposition 5.3]{Katzm08}).
\item If $\K^{1/p}$ is a finite extension of $\K$, then $
\K[[X_1,\ldots,X_n]][Y_1,\ldots,Y_m]$ is $F$-$\cap$-flat, because $T
:= \K[[X_1,\ldots,X_n]]$ is such that $T^{1/p}$ is a free
$T$-module. See \cite[p.\ 941]{Katzm08}.
\end{enumerate}
\end{exs}

It is natural to ask whether there are examples of regular rings of
prime characteristic that are not $F$-$\cap$-flat. I am very
grateful to M. Hochster for permitting me to include the following.

\begin{ex}[M. Hochster] \label{mh} Let $\K$ be an algebraically closed
field of characteristic $p$ and let $T_1,\ldots,T_n,\ldots$ be
countably many independent indeterminates over $\K$. Let
$$
L = K(T_1,\ldots,T_n,\ldots), \quad L_0 =
K(T_1^p,\ldots,T_n^p,\ldots)
$$
and $L_i = L_0(T_1,\ldots,T_i) = K(T_1,\ldots,T_i,T_{i+1}^p,
\ldots,T_{i+j}^p, \ldots )$ for all $i \in \N$. Let $X,Y$ be
independent indeterminates over $L$, and let $R = \bigcup_{i\in
\nn}L_i[[X,Y]]$, a subring of $L[[X,Y]]$. It is straightforward to
show that $R$ is a Noetherian local ring (Cohen's Theorem can be
helpful) that is regular of dimension $2$. In fact, working inside
an algebraic closure of the quotient field $\Quot (L[[X,Y]])$ of
$L[[X,Y]]$, we have $R \subseteq \widehat{R} = L[[X,Y]] \subseteq
R^{1/p} \subseteq \widehat{R}^{1/p}$, from which we deduce that
$R^{1/p}$ is a faithfully flat extension of $L[[X,Y]]$ (because
$\widehat{R}^{1/p}$ is faithfully flat over both $R^{1/p}$ and
$\widehat{R}$).

Let $f_k := Y - \sum_{i=1}^{k} T_iX^i$ and $\fa_k := (f_k,
X^{k+1})R$, for all $k \in \N$. In $L[X,Y]]$, set $f := Y -
\sum_{i=1}^{\infty} T_iX^i$. Note that $\fa_kL[[X,Y]] =
(f,X^{k+1})L[[X,Y]]$ and so $\fa_kR^{1/p} = (f,X^{k+1})R^{1/p}$, for
all $k \in \N$. Krull's Intersection Theorem therefore yields that
$\bigcap_{k\in \N}\left(\fa_kL[[X,Y]]\right) = fL[[X,Y]]$ and
$\bigcap_{k\in \N}\left(\fa_kR^{1/p}\right) = fR^{1/p}$. By faithful
flatness, $fR^{1/p} \cap L[[X,Y]] = fL[[X,Y]]$, a prime ideal of
$L[[X,Y]]$ of height $1$; therefore $fR^{1/p} \cap R$ is a prime
ideal of $R$ of height $1$ or $0$, and so is principal, generated by
$g \in R$, say. Thus $gR = fR^{1/p} \cap R = {\textstyle
\bigcap_{k\in \N}\left(\fa_kR^{1/p}\right)} \cap R =
{\textstyle\bigcap_{k\in \N}\fa_k}$.

Suppose that $R$ is $F$-$\cap$-flat. Then
\begin{align*}
gL[[X,Y]] & = \left(\left({\textstyle\bigcap_{k\in
\N}\fa_k}\right)R^{1/p}\right) \cap L[[X,Y]] = {\textstyle
\bigcap_{k\in \N}\left(\fa_kR^{1/p}\right)} \cap L[[X,Y]]\\ & =
{\textstyle \bigcap_{k\in \N}\fa_kL[[X,Y]]} = fL[[X,Y]].
\end{align*}
Thus $f = gu$ for some unit $u \in L[[X,Y]]$. There exists $j \in
\N$ such that $g \in L_j[[X,Y]]$. After multiplication by an element
of $L_j$, we can assume that the homogeneous component of $g$ of
degree $1$ is $Y -T_1X$. Use the division algorithm (for division by
$Y-T_1X \in L_j[X][Y]$) in order to construct, inductively, the
homogeneous components of a unit $v$ of $L_j[[X,Y]]$ such that $vg =
Y - h$ for some non-unit $h \in L_j[[X]]$. Then $f = (Y-h)uv^{-1}$
in $L[[X,Y]]$, and the substitution $X \mapsto X$, $Y \mapsto h$
therefore yields that $h = \sum_{i=1}^{\infty} T_iX^i$. This is a
contradiction since $h \in L_j[[X]]$.

The regular local ring $R$ in this example is not excellent, because
$\Quot(\widehat{R})$ is a purely inseparable extension of
$\Quot(R)$. An interesting question, believed by this author to be
open, is whether there exists an {\em excellent\/} regular ring of
prime characteristic that is not $F$-$\cap$-flat.
\end{ex}

\begin{lem}\label{if.3} Let $S$ be a regular ring of characteristic
$p$ which is $F$-$\cap$-flat. Let
$(\fa_{\lambda})_{\lambda\in\Lambda}$ be a non-empty family of
ideals of $S$ and let $n \in \nn$. Then
$\bigcap_{\lambda\in\Lambda}(\fa_{\lambda}^{[p^n]}) =
\left(\bigcap_{\lambda\in\Lambda}\fa_{\lambda}\right)^{[p^n]}$.

Let $\fb$ be an ideal of $S$. Then there exists a smallest ideal
$\ft$ of $S$ such that $\fb \subseteq \ft^{[p^n]}$ (`smallest' in
the sense that $\ft \subseteq \fh$ for each ideal $\fh$ of $S$ for
which $\fb \subseteq \fh^{[p^n]}$).
\end{lem}

\begin{proof} As observed in \ref{if.1}, the $n$th iterate $f^n : S
\lra S$ is intersection-flat. Consider the $n$th component $Sy^n$ of
the Frobenius skew polynomial ring $S[y,f]$ over $S$ (in the
indeterminate $y$), and let $\fa$ be an ideal of $S$. Then, when the
left $S$-module $Sy^n \otimes_SS$ is identified with $S$ in the
natural way, the image of the $S$-monomorphism $Sy^n\otimes_S\fa
\lra Sy^n\otimes_SS$ induced by inclusion becomes identified with
$\fa^{[p^n]}$. The first claim therefore follows from the definition
of intersection-flatness, while the second claim follows from
application of the first claim to the (non-empty) family of ideals
$\fh$ of $S$ for which $\fb \subseteq \fh^{[p^n]}$ (note that $\fb
\subseteq S^{[p^n]}$).
\end{proof}

\begin{defi}\label{if.4} Let $S$ be a regular ring of characteristic
$p$, let $n \in \nn$ and let $\fa$ be an ideal of $S$. We shall say
that {\em $\fa^{[1/p^n]}$ exists\/} if there is a smallest ideal
$\ft$ of $S$ (which will be denoted by $\fa^{[1/p^n]}$) such that
$\fa \subseteq \ft^{[p^n]}$. We shall sometimes refer to
$\fa^{[1/p^n]}$ (when it exists) as the {\em $p^n$th Frobenius root
of $\fa$}.

Note that $\fa^{[1/p^0]}$ always exists, and is equal to $\fa$. If
$S$ is $F$-$\cap$-flat, then $\fa^{[1/p^n]}$ exists for all $n \in
\nn$ and all ideals $\fa$ of $S$ (by Lemma \ref{if.3}).
\end{defi}

The notation $\fa^{[1/p^n]}$  of \ref{if.4} is consistent with that
used by Blickle, M. Mustata and Smith, in \cite[\S 2]{BlMuSm08},
where they used $p^n$th Frobenius roots, in the case of $F$-finite
regular rings of characteristic $p$, in connection with
generalizations of test ideals studied by N. Hara and K. Yoshida
\cite[\S 6]{HarYos03} and Hara and S. Takagi \cite[\S 1]{HarTak04}.
Other places in which $p^n$th Frobenius roots have appeared in
similar $F$-finite contexts include Takagi \cite[Proposition
2.5]{Takag10}, K. Schwede \cite[Remark 3.8]{Schwe09}, and Blickle,
Schwede, Takagi and W. Zhang \cite[Proposition 3.10]{BSTZ10}.

\begin{rmk}\label{if.5} Let $S$ be a regular ring of characteristic
$p$, let $(\fa_{\lambda})_{\lambda\in\Lambda}$ be a non-empty family
of ideals of $S$, and let $(q_{\lambda})_{\lambda\in\Lambda}$ be a
family of powers of $p$ (with non-negative exponents), indexed by
the same set $\Lambda$.

If $\fa_{\lambda}^{[1/q_{\lambda}]}$ exists for all
$\lambda\in\Lambda$, then
$\sum_{\lambda\in\Lambda}\fa_{\lambda}^{[1/q_{\lambda}]}$ is the
smallest ideal $\fh$ of $S$ such that $\fa_{\lambda} \subseteq
\fh^{[q_{\lambda}]}$ for all $\lambda\in\Lambda$.

For example, if $\fa,\fb$ are ideals of $S$ and $n \in \N$ is such
that $\fb^{[1/p^n]}$ exists, then $\fb^{[1/p^n]} + \fa^{[1/p^0]} =
\fb^{[1/p^n]} + \fa$ is the smallest ideal $\ft$ of $S$ which
contains $\fa$ and is such that $\fb \subseteq \ft^{[p^n]}$.
\end{rmk}

\begin{lem}\label{if.6} Let $S$ be a regular ring of characteristic
$p$, let $n \in \nn$, let $\fa$ be an ideal of $S$ and let $W$ be a
multiplicatively closed subset of $S$. If $\fa^{[1/p^n]}$ exists,
then so also does $(\fa W^{-1}S)^{[1/p^n]}$, and $(\fa
W^{-1}S)^{[1/p^n]} = \fa^{[1/p^n]}W^{-1}S$.
\end{lem}

\begin{proof} Set $\ft := \fa^{[1/p^n]}$; since $\fa \subseteq
\ft^{[p^n]}$, we have $\fa W^{-1}S\subseteq \ft^{[p^n]}W^{-1}S =
(\ft W^{-1}S)^{[p^n]}$.

Now let $\fB$ be an ideal of $W^{-1}S$ such that $\fa
W^{-1}S\subseteq \fB^{[p^n]}$. Let\/ $\phantom{}^e$ and
$\phantom{}^c$ denote extension and contraction of ideals under the
natural ring homomorphism $S \lra W^{-1}S$. Since every ideal of
$W^{-1}S$ is extended from $S$, we can write $\fB = \fb^e$ where
$\fb = \fB^c$. Let $\fb = \fq_1 \cap \ldots \cap \fq_t$ be a minimal
primary decomposition of $\fb$, and note that all of $\fq_1, \ldots,
\fq_t$ are disjoint from $W$. It follows from the faithful flatness
of the Frobenius homomorphism on $S$ that $\fb^{[p^n]} =
\fq_1^{[p^n]} \cap \cdots \cap \fq_t^{[p^n]}$ is a minimal primary
decomposition of $\fb^{[p^n]}$, so that
$$
\fB^{[p^n]}  = (\fb^e)^{[p^n]} = (\fb^{[p^n]})^e =
\left(\fq_1^{[p^n]} \cap \cdots \cap \fq_t^{[p^n]}\right)^e
=\left(\fq_1^{[p^n]}\right)^e \cap \cdots \cap
\left(\fq_t^{[p^n]}\right)^e.
$$
Since $\fq_i^{[p^n]}$ is $\sqrt{\fq_i}$-primary for all $i = 1,
\ldots, t$, we see that
$$
\fa  \subseteq \fa^{ec} \subseteq (\fB^{[p^n]})^c  =
\left(\fq_1^{[p^n]}\right)^{ec} \cap \cdots \cap
\left(\fq_t^{[p^n]}\right)^{ec} = \fq_1^{[p^n]} \cap \cdots \cap
\fq_t^{[p^n]} = \fb^{[p^n]}.
$$
Therefore $\ft \subseteq \fb$ and $\ft^e \subseteq \fb^e = \fB$.  It
follows that $(\fa W^{-1}S)^{[1/p^n]}$ exists and is equal to $\ft^e
= (\fa^{[1/p^n]})^e$.
\end{proof}

We are going to use the following lemma of G. Lyubeznik and K. E.
Smith. Their proof in \cite[Lemma 6.6]{LyuSmi01} uses work of N.
Radu \cite[Corollary 5]{Radu92}, which, in turn, uses D. Popescu's
general N\'eron desingularization \cite{Popes85, Popes86}. (For an
exposition of Popescu's work, see R. G. Swan \cite{Swan95}.)

\begin{lem} [G. Lyubeznik and K. E. Smith {\cite[Lemma 6.6]{LyuSmi01}}]
\label{if.7} Let $S$ be an excellent regular local ring of
characteristic $p$, and let $\fB$ be an ideal of the completion
$\widehat{S}$ of $S$. Then $\fB^{[p^n]}\cap S = (\fB \cap
S)^{[p^n]}$ for all $n \in \nn$.
\end{lem}

\begin{lem}\label{if.8} Let $S$ be an excellent regular local ring of characteristic
$p$, let $n \in \nn$ and let $\fa$ be an ideal of $S$. If
$\fa^{[1/p^n]}$ exists, then so also does
$(\fa\widehat{S})^{[1/p^n]}$, and $(\fa\widehat{S})^{[1/p^n]} =
\fa^{[1/p^n]}\widehat{S}$.
\end{lem}

\begin{proof} Set $\ft := \fa^{[1/p^n]}$; since $\fa \subseteq
\ft^{[p^n]}$, we have $\fa \widehat{S}\subseteq
\ft^{[p^n]}\widehat{S} = (\ft \widehat{S})^{[p^n]}$.

Now let $\fB$ be an ideal of $\widehat{S}$ such that $\fa
\widehat{S}\subseteq \fB^{[p^n]}$. Then, on use of the Lemma
\ref{if.7} of Lyubeznik and Smith, we see that
$$
\fa = \fa \widehat{S} \cap S \subseteq \fB^{[p^n]}\cap S = (\fB \cap
S)^{[p^n]}.$$ Therefore $\ft \subseteq \fB \cap S$, so that
$\ft\widehat{S} \subseteq (\fB \cap S)\widehat{S} \subseteq \fB$. It
follows that $(\fa\widehat{S})^{[1/p^n]}$ exists and is equal to
$\ft\widehat{S} = \fa^{[1/p^n]}\widehat{S}$.
\end{proof}

Use will be made in this section of the following extension, due to
G. Lyubeznik, of a result of R. Hartshorne and R. Speiser. It shows
that, when $R$ is local, an $x$-torsion left $R[x,f]$-module which
is Artinian (that is, `cofinite' in the terminology of Hartshorne
and Speiser) as an $R$-module exhibits a certain uniformity of
behaviour.

\begin{thm} [G. Lyubeznik {\cite[Proposition 4.4]{Ly}}] \label{HSL}
{\rm (Compare also Hartshorne and Speiser \cite[Proposition
1.11]{HS}.)} Suppose that $(R,\m)$ is local and let $H$ be a left
$R[x,f]$-module which is Artinian as an $R$-module. Then there
exists $e \in \nn$ such that $x^e\Gamma_x(H) = 0$.
\end{thm}

Hartshorne and Speiser first proved this result in the case where
$R$ is local and contains its residue field which is perfect.
Lyubeznik applied his theory of $F$-modules to obtain the result
without restriction on the local ring $R$ (of characteristic $p$). A
short proof of the theorem, in the generality achieved by Lyubeznik,
is provided in \cite{HSLonly}.

\begin{defi}
\label{hslno} Suppose that $(R,\fm)$ is local, and let $H$ be a left
$R[x,f]$-module which is Artinian as an $R$-module.  By the
Hartshorne--Speiser--Lyubeznik Theorem \ref{HSL}, there exists $e
\in \nn$ such that $x^e\Gamma_x(H) = 0$: we call the smallest such
$e$ the {\em Hartshorne--Speiser--Lyubeznik number\/}, or {\em
HSL-number\/} for short, of $H$, and denote this by $\HSL(H)$.
\end{defi}

\begin{rmk}\label{dualHSL} The referee has suggested that I point
out that there is a result about {\em right\/} $R[x,f]$-modules that
is, in a sense, dual to the Hartshorne--Speiser--Lyubeznik Theorem:
if $M$ is a right $R[x,f]$-module that is Noetherian as an
$R$-module, then there exists $e \in \nn$ such that $Mx^e =
Mx^{e+1}$. This follows from M. Blickle and G. Boeckle
\cite[Proposition 2.14]{BB}; they prove their result via an argument
of O. Gabber \cite[Lemma 13.1]{Gabber}. There is another proof in
\cite[Theorem 3.4]{ShaYos11}. Neither result requires that $R$ be
$F$-finite, but in the case where $(R,\fm)$ is $F$-finite, local and
complete, the Blickle--Boeckle result can be deduced from the
Hartshorne--Speiser--Lyubeznik Theorem via Matlis duality.
\end{rmk}

\begin{prop}\label{if.11} Let $(S,\fn)$ be an excellent regular local ring of characteristic $p$,
and let $\fa$ be a proper ideal of $S$; let $R := S/\fa$, and denote
the maximal ideal of $R$ by $\fm$. Set $E := E_S(S/\fn)$, and let $u
\in (\fa^{[p]}:\fa)$. Use $u$ to furnish $E_R(R/\fm) = (0:_E\fa)$
with a structure as a left $R[x,f]$-module, and recall that
$E_R(R/\fm)$ is Artinian as an $R$-module. Assume that
$(Sdu^j)^{[1/p^n]}$ exists for all $n,j\in\nn$ and $d\in S$. Recall
that $\omega_0 := 0$ and $\omega_n := 1 + p + \cdots + p^{n-1}$ for
$n \in \N$. Let $e \in \nn$. Then

\begin{enumerate}
\item $(0:_{E_R(R/\fm)}x^e) = \left(0:_E(Su^{\omega_e})^{[1/p^e]}+\fa\right)$;
\item for $d \in S$, $$ \left\{h \in E_R(R/\fm) : dx^nh= 0 \text{~for all~} n \geq e\right\} =
{\textstyle \left(0:_E \fa + \sum_{n\geq e}
(Sdu^{\omega_n})^{[1/p^n]}\right)}\mbox{;}$$
\item $(Su^{\omega_n})^{[1/p^n]}+\fa \supseteq
(Su^{\omega_{n+1}})^{[1/p^{n+1}]}+\fa$ for all $n \in \nn$;
\item if $(Su^{\omega_e})^{[1/p^e]}+\fa = (Su^{\omega_{e+1}})^{[1/p^{e+1}]}+\fa$,
then $$(Su^{\omega_e})^{[1/p^e]}+\fa =
(Su^{\omega_{e+j}})^{[1/p^{e+j}]}+\fa \quad \text{for all~} j \in
\N\mbox{; and}$$
\item $\HSL(E_R(R/\fm))$ is equal to the least integer $n \in \nn$
such that $(Su^{\omega_n})^{[1/p^n]}+\fa =
(Su^{\omega_{n+1}})^{[1/p^{n+1}]}+\fa$.
\end{enumerate}
\end{prop}

\begin{note} Parts of this proposition, particularly (iii), (iv) and
(v), were inspired by work of M. Katzman in \cite[Theorem
4.6]{Katzm08}. However, in that cited theorem of Katzman, the local
rings were complete, and that is not necessarily the case here.
\end{note}

\begin{proof} The completion $\widehat{R}$ of $R = S/\fa$ is given
by $\widehat{R} = \widehat{S}/\fa\widehat{S}$. Now the $R$-module
structure on $E_R(R/\fm)$ can be extended to a structure as an
$\widehat{R}$-module; as such, $E_R(R/\fm) =
E_{\widehat{R}}(\widehat{R}/\fm\widehat{R})$. Note that $u \in
(\fa^{[p]}:\fa)\widehat{S} = ((\fa\widehat{S})^{[p]}:
\fa\widehat{S})$: if we use $u$ to put a left
$\widehat{R}[x,f]$-module structure on
$E_{\widehat{R}}(\widehat{R}/\fm\widehat{R})$, then this structure
extends the left $R[x,f]$-module structure on $E_R(R/\fm)$.

We can use Lemma \ref{if.8} to see that
$(\widehat{S}du^{j})^{[1/p^n]}$ exists and is equal to
$(Sdu^j)^{[1/p^n]}\widehat{S}$, for all $n,j\in\nn$ and $d\in S$.

It follows from Matlis duality (see \cite[p.\ 154]{SV}, for example)
that each $\widehat{S}$-submodule $M$ of $E$ satisfies $M =
(0:_E(0:_{\widehat{S}}M))$, and so has the form $(0:_E\fA)$ for a
(uniquely determined) ideal $\fA$ of $\widehat{S}$.

(i) Now $(0:_{E_R(R/\fm)}x^e) =
(0:_{E_{\widehat{R}}(\widehat{R}/\fm\widehat{R})}x^e)$ is an
$\widehat{S}[x,f]$-submodule of $E$, and so has the form $(0:_E\fB)$
for an ideal $\fB$ of $\widehat{S}$ that contains $\fa\widehat{S}$.
It follows from \cite[Proposition 2.7]{Fpurhastest} that
$\grann_{\widehat{S}[x,f]}(0:_E\fB) =
\bigoplus_{n\in\nn}(\fB^{[p^n]}:u^{\omega_n})x^n.$ We therefore
deduce that $\fB$ is the smallest ideal of $\widehat{S}$ such that
$\fa \widehat{S} \subseteq \fB$ and $u^{\omega_e} \in \fB^{[p^e]}$,
that is,
$$\fB = (\widehat{S}u^{\omega_{e}})^{[1/p^{e}]} + \fa\widehat{S} =
(Su^{\omega_e})^{[1/p^e]}\widehat{S} + \fa\widehat{S} =
((Su^{\omega_e})^{[1/p^e]} + \fa)\widehat{S}.$$ Therefore
$(0:_{E_R(R/\fm)}x^e) = (0:_E((Su^{\omega_e})^{[1/p^e]} +
\fa)\widehat{S}) = (0:_E(Su^{\omega_e})^{[1/p^e]} + \fa)$.

(ii) The proof of this is similar to the above proof of part (i).

Let $ L := \left\{h \in E_R(R/\fm) : dx^nh= 0 \text{~for all~} n
\geq e\right\}$. Then $L$ is an $\widehat{S}[x,f]$-submodule of $E$,
and so has the form $(0:_E\fC)$ for an ideal $\fC$ of $\widehat{S}$
that contains $\fa\widehat{S}$. By \cite[Proposition
2.7]{Fpurhastest} again, $\grann_{\widehat{S}[x,f]}(0:_E\fC) =
\bigoplus_{n\in\nn}(\fC^{[p^n]}:u^{\omega_n})x^n,$ so that $\fC$ is
the smallest ideal of $\widehat{S}$ such that $\fa \widehat{S}
\subseteq \fC$ and $d \in (\fC^{[p^n]}:u^{\omega_n})$ for all $n
\geq e$. Thus, in view of Remark \ref{if.5},
\begin{align*}
\fC &=  \fa\widehat{S} + {\textstyle \sum_{n\geq e}
(\widehat{S}du^{\omega_n})^{[1/p^n]}}\\ & = \fa\widehat{S} +
{\textstyle \sum_{n\geq e} (Sdu^{\omega_n})^{[1/p^n]}\widehat{S}} =
{\textstyle \left(\fa + \sum_{n\geq e}
(Sdu^{\omega_n})^{[1/p^n]}\right)}\widehat{S}.
\end{align*}
Therefore $L = {\textstyle \left(0:_E \left(\fa + \sum_{n\geq e}
(Sdu^{\omega_n})^{[1/p^n]}\right)\widehat{S}\right)} = {\textstyle
\left(0:_E \fa + \sum_{n\geq e} (Sdu^{\omega_n})^{[1/p^n]}\right)}$,
as required.

(iii) Let $n \in \nn$ and $\ft := (Su^{\omega_n})^{[1/p^n]}$. Thus
$u^{\omega_n} \in \ft^{[p^n]}$. Therefore $u^{\omega_np} \in
\ft^{[p^{n+1}]}$, so that $u^{\omega_{n+1}} = uu^{\omega_np} \in
\ft^{[p^{n+1}]}$. Hence $\ft \supseteq
(Su^{\omega_{n+1}})^{[1/p^{n+1}]}$.

(iv) Suppose that $(Su^{\omega_e})^{[1/p^e]}+\fa =
(Su^{\omega_{e+1}})^{[1/p^{e+1}]}+\fa$. By part (i), this implies
that $(0:_{E_R(R/\fm)}x^e) = (0:_{E_R(R/\fm)}x^{e+1})$. It follows
from this that $(0:_{E_R(R/\fm)}x^e) = (0:_{E_R(R/\fm)}x^{e+j})$ for
all $j\in \N$, so that $(Su^{\omega_e})^{[1/p^e]}+\fa =
(Su^{\omega_{e+j}})^{[1/p^{e+j}]}+\fa$ for all $j \in \N$ by part
(i) (since an $S$-module and its Matlis dual have equal
annihilators, even if $S$ is not complete).

(v) Let $h := \HSL(E_R(R/\fm))$; this means that
$\Gamma_x(E_R(R/\fm)) = (0:_{E_R(R/\fm)}x^h) =
(0:_{E_R(R/\fm)}x^{h+j})$ for all $j \in \N$, and that
$(0:_{E_R(R/\fm)}x^h) \supset (0:_{E_R(R/\fm)}x^{h-1})$ if $h > 0$.
(The symbol `$\supset$' is reserved to denote strict containment.)
It now follows from parts (i), (ii) and (iii) (and the fact that an
$S$-module and its Matlis dual have equal annihilators) that
$(Su^{\omega_{h}})^{[1/p^{h}]}+\fa =
(Su^{\omega_{h+j}})^{[1/p^{h+j}]}+\fa$ for all $j \in \N$, whereas
$(Su^{\omega_n})^{[1/p^n]}+\fa \supset
(Su^{\omega_{n+1}})^{[1/p^{n+1}]}+\fa$ for all $n \in\nn$ with $n <
h$.
\end{proof}

We can use Proposition \ref{if.11} to establish a bound on
$\HSL$-numbers in some circumstances.

\begin{thm}\label{if.12} Suppose that $R$ is a homomorphic image
$S/\fa$ of an excellent regular ring $S$ (of characteristic $p$)
modulo a proper ideal $\fa$, and assume that $S$ is $F$-$\cap$-flat.
Let $u \in (\fa^{[p]}:\fa)$.

Given a $\fq \in \Var(\fa)$, write $\fp = \fq/\fa$, let $k(\fp)$
denote the simple $R_{\fp}$-module and use $u/1 \in
(\fa^{[p]}:\fa)S_{\fq} = ((\fa S_{\fq})^{[p]}:\fa S_{\fq})$ to
furnish $E_{R_{\fp}}(k(\fp))$ with a structure as a left
$R_{\fp}[x,f]$-module.

Then there exists $h \in \N$ such that $\HSL(E_{R_{\fp}}(k(\fp)))
\leq h$ for all\/ $\fp \in \Spec (R)$.
\end{thm}

\begin{note} I am very grateful to Mordechai Katzman for pointing
out to me that the quasi-compactness of $\Var(\fa)$ in the Zariski
topology on $\Spec(S)$ can be used to prove this theorem.
\end{note}

\begin{proof} It follows from Lemma \ref{if.3} that $\fb^{[1/p^n]}$
exists for all $n\in\nn$ and all ideals $\fb$ of $S$. Lemma
\ref{if.6} shows that $(\fb S_{\fq})^{[1/p^n]}$ exists and is equal
to $\fb^{[1/p^n]}S_{\fq}$ for all $n\in\nn$, $\fq \in \Spec(S)$ and
ideals $\fb$ of $S$.

Set $\ft_n := (Su^{\omega_n})^{[1/p^n]}+\fa$ (where $\omega_n$ is as
in \ref{jun.3}) for all $n \in \nn$. One checks easily that $\ft_n
\supseteq \ft_{n+1}$ by use of the argument in the proof of
Proposition \ref{if.11}(iii). It follows from the first paragraph of
this proof and Proposition \ref{if.11} that, for a $\fq \in
\Var(\fa)$, if $\ft_nS_{\fq} = \ft_{n+1}S_{\fq}$ for an $n \in \nn$,
then $\ft_nS_{\fq} = \ft_{n+j}S_{\fq}$ for all $j\in\N$, and that
$\HSL(E_{R_{\fq/\fa}}(k(\fq/\fa)))$ is the smallest integer $h'$
such that $\ft_{h'}S_{\fq} = \ft_{h'+1}S_{\fq}$.

For each $n\in\nn$, let $C_n := \Supp(\ft_n/\ft_{n+1})$, a closed
subset of $\Var(\fa)$ in the Zariski topology. The facts listed in
the immediately preceding paragraph show that $C_n \supseteq
C_{n+1}$ for all $n\in\nn$, and that $\bigcap_{n\in\nn}C_n =
\emptyset$. Since $\Var(\fa)$ is quasi-compact, there exists $h \in
\N$ such that $C_h = \emptyset$. Then $\ft_{h}S_{\fq} =
\ft_{h+1}S_{\fq}$ for all $\fq \in \Var(\fa)$, so that, by
Proposition \ref{if.11}(iv),(v), we have $\HSL(E_{R_{\fp}}(k(\fp)))
\leq h$ for all $\fp \in \Spec (R)$.
\end{proof}

\section{\it Big test elements for certain non-local $(R_0)$ rings}
\label{nlrn}

The aims of this section are to establish the existence of big test
elements for $R$ when $R$ satisfies condition $(R_0)$ and is a
homomorphic image of an excellent regular ring of characteristic $p$
that is $F$-$\cap$-flat, and then to use faithfully flat ring
extensions to extend the scope of this approach. The proof will make
use of the bound on $\HSL$-numbers established in Theorem
\ref{if.12}.

\begin{lem}\label{nlrn.1} Suppose that $(R,\fm)$ is local and that
$H$ is a left $R[x,f]$-module that is Artinian as an $R$-module. Let
$c \in R$ and, for $m, t \in \nn$ with $m \geq t \geq \HSL(H)$, set
$L_{m,t} := \left\{h \in H : c^{p^m}x^nh= 0 \text{~for all~} n \geq
t\right\}$. Then $L_{m,t} = L_{m+1,t} = L_{m+1,t+1}$.

Consequently, $\left(0:_R\ann_H(\bigoplus_{n\geq
t}Rc^{p^{m+j}}x^n)\right) = \left(0:_R\ann_H(\bigoplus_{n\geq
t}Rc^{p^{m}}x^n)\right)$ for all $j \in \nn$.
\end{lem}

\begin{proof} It is clear that $L_{m,t} \subseteq L_{m+1,t} \subseteq
L_{m+1,t+1}$, so let $h \in L_{m+1,t+1}$. Thus, for all $n \geq
t+1$, we have
$$
0 = c^{p^{m+1}}x^nh = (c^{p^{m-t}})^{p^{t+1}}x^{t+1}x^{n-(t+1)}h =
x^{t+1}c^{p^{m-t}}x^{n-(t+1)}h.
$$
Therefore $x^{t}c^{p^{m-t}}x^{n-(t+1)}h = 0$, since $t \geq
\HSL(H)$; therefore $(c^{p^{m-t}})^{p^{t}}x^{t}x^{n-(t+1)}h = 0$,
that is, $c^{p^m}x^{n-1}h= 0$. As this is true for all $n \geq t+1$,
we see that $h \in L_{m,t}$.

The final claim is now immediate.
\end{proof}

\begin{thm}
\label{nlrn.2} Suppose that $R$ is a homomorphic image $S/\fa$ of an
excellent regular ring $S$ of characteristic $p$ modulo a proper
ideal $\fa$, and assume that $S$ is $F$-$\cap$-flat and that $R$
satisfies condition $(R_0)$. Then $R$ has a big test element.

In fact, each $c \in R^{\circ}$ for which $R_c$ is Gorenstein and
weakly $F$-regular has a power that is a big test element for $R$.
\end{thm}

\begin{proof} We can assume that $\fa \neq 0$, since otherwise the
claims are clear (by Corollary \ref{may.2}, for example). By
Proposition \ref{tc.postant}, it is enough for us to prove the final
claim under the additional hypothesis that $R$ is an integral
domain, and so we suppose that $\fa$ is a non-zero prime ideal of
$S$.

Choose a $\fq \in \Var(\fa)$ such that $R_{\fq/\fa}$ is Gorenstein
and weakly $F$-regular (and therefore $F$-pure by
\ref{tcjalg.1}(i)). (For example, one could take $\fq = \fa$.)
Choose $u_1, \ldots, u_l \in (\fa^{[p]}:\fa) \setminus \fq^{[p]}$
(whose natural images in $T:= (\fa^{[p]}:\fa)/\fa^{[p]}$ generate
$T$) as in Proposition \ref{jun.1}. Since $S_{\fq}/\fa S_{\fq} \cong
R_{\fq/\fa}$ is a Gorenstein local ring, it follows from
\ref{jun.5}(iii) that the $S_{\fq}$-module $\left((\fa
S_{\fq})^{[p]} : \fa S_{\fq}\right)/(\fa S_{\fq})^{[p]}$ is cyclic,
so there is $i \in \{1,\ldots,l\}$ such that $ \left((\fa
S_{\fq})^{[p]} : \fa S_{\fq}\right) =  (\fa S_{\fq})^{[p]} +
(u_i/1)S_{\fq}. $ Let $u := u_i$.

Given a general $\fq' \in \Var(\fa)$, write $\fp' = \fq'/\fa$, let
$k(\fp')$ (respectively $k(\fq')$) denote the simple
$R_{\fp'}$-module (respectively the simple $S_{\fq'}$-module) and
use $u/1 \in (\fa^{[p]}:\fa)S_{\fq'} = ((\fa S_{\fq'})^{[p]}:\fa
S_{\fq'})$ to furnish $E_{R_{\fp'}}(k(\fp')) =
(0:_{E_{S_{\fq'}}(k(\fq'))}\fa S_{\fq'})$ with a structure as a left
$R_{\fp'}[x,f]$-module. Let $h \in \N$ be such that
$\HSL(E_{R_{\fp'}}(k(\fp'))) \leq h$ for all $\fp' \in \Spec (R)$.
(Theorem \ref{if.12} guarantees the existence of such an $h$.)

Let $c'' \in R^{\circ} = R \setminus \{0\}$ be such that $R_{c''}$
is regular. Choose $d \in S$ whose natural image in $R$ is equal to
$c''$. Thus $d \not\in\fa$. It follows from Lemma \ref{if.3} that
$\fb^{[1/p^n]}$ exists for all $n\in\nn$ and all ideals $\fb$ of
$S$. Let $\ft := \fa + \sum_{n\geq h}
(Sd^{p^h}u^{\omega_n})^{[1/p^n]}$. The comments in Remark \ref{if.5}
show that $\ft$ is the smallest ideal of $S$ such that $\fa
\subseteq \ft$ and $d^{p^h}u^{\omega_n} \in \ft^{[p^n]}$ for all $n
\geq h$. We suppose that $\height_R(\ft/\fa) = 0$, that is, that
$\ft = \fa$, and seek a contradiction.

Our supposition means that $d^{p^h}u^{\omega_n} \in \fa^{[p^n]}$ for
all $n \geq h$. It follows from the faithful flatness of the
Frobenius homomorphism on $S$ that the ideal $\fa^{[p^n]}$ is
$\fa$-primary for all $n \in \N$. Since $d \not\in \fa$, we see that
$u^{\omega_n} \in \fa^{[p^n]}$ for all $n \geq h$. Now $u \in
(\fa^{[p]}:\fa)$; therefore $u \in \fa^{[p]}$, by Lemma \ref{jun.2}.
This is a contradiction to the choice of $u$. We have thus proved
that $\height_R(\ft/\fa)
> 0$.

By Lemma \ref{if.6}, we have $\ft S_{\fq'} = \fa S_{\fq'} +
\sum_{n\geq h} (S_{\fq'}(d/1)^{p^h}(u/1)^{\omega_n})^{[1/p^n]}$,
and, by Remark \ref{if.5}, this is the smallest ideal $\fB$ of
$S_{\fq'}$ such that $\fa S_{\fq'} \subseteq \fB$ and
$(d/1)^{p^h}(u/1)^{\omega_n} \in \fB^{[p^n]}$ for all $n \geq h$. By
Proposition \ref{if.11}(ii), we have
$$
\left\{\xi \in E_{R_{\fp'}}(k(\fp')) : (c''/1)^{p^h}x^n\xi= 0
\text{~for all~} n \geq h\right\} = \left(0:_{E_{S_{\fq'}}(k(\fq'))}
\ft S_{\fq'}\right).
$$

We now apply the Embedding Theorem \ref{pl.5b} to the
$R_{\fp'}$-module $E_{R_{\fp'}}(k(\fp'))$, with the graded companion
$\widetilde{E_{R_{\fp'}}(k(\fp'))}$ playing the r\^ole of $H$. The
conclusion is that there is a family $\left(L^{(n)}\right)_{n \in
\nn}$ of\/ $\nn$-graded left $R_{\fp'}[x,f]$-modules, where
$L^{(n)}$ is an $n$-place extension of the $-n$th shift of a graded
product of copies of $\widetilde{E_{R_{\fp'}}(k(\fp'))}$ (for each
$n \in \nn$), for which there exists a homogeneous
$R_{\fp'}[x,f]$-monomorphism
\[
\nu : R_{\fp'}[x,f]\otimes_{R_{\fp'}}E_{R_{\fp'}}(k(\fp')) =
\bigoplus_{i\in
\nn}(R_{\fp'}x^i\otimes_{R_{\fp'}}E_{R_{\fp'}}(k(\fp'))) \lra
\prod_{n\in\nn}{\textstyle ^{^{^{\!\!\Large \prime}}}} L^{(n)} =: K.
\]
Since $\big(R_{\fp'}\big)_{c''/1}$, the ring of fractions of
$R_{\fp'}$ with respect to the set of powers of the element $c''/1$
of $(R_{\fp'})^{\circ}$, is a ring of fractions of $R_{c''}$, it is
regular. Note also that $c''/1 \in (R_{\fp'})^{\circ}$. By Theorem
\ref{jan.17}, there is a $j'\in\N$ such that $(c''/1)^{j'}$ is a big
test element for $R_{\fp'}$; hence there exists $j \in \N$ with $j
\geq h$ such that $(c''/1)^{p^j}$ is a big test element for
$R_{\fp'}$, so that
$$\Delta_{R_{\fp'}}\!\left(R_{\fp'}[x,f]\otimes_{R_{\fp'}}E_{R_{\fp'}}(k(\fp'))\right)
\subseteq
\ann_{R_{\fp'}[x,f]\otimes_{R_{\fp'}}E_{R_{\fp'}}(k(\fp'))}\left({\textstyle
\bigoplus_{n\geq h}R_{\fp'}(c''/1)^{p^j}x^n}\right).
$$
Thus
$\nu\left(\Delta_{R_{\fp'}}\!\left(R_{\fp'}[x,f]\otimes_{R_{\fp'}}E_{R_{\fp'}}(k(\fp'))\right)\right)
\subseteq \ann_{K}\left({\textstyle \bigoplus_{n\geq
h}R_{\fp'}(c''/1)^{p^j}x^n}\right), $ and so
$$
\left(0:_{R_{\fp'}}\ann_{K}\left({\textstyle \bigoplus_{n\geq
h}R_{\fp'}(c''/1)^{p^j}x^n}\right)\right) \subseteq
\left(0:_{R_{\fp'}}\Delta_{R_{\fp'}}\!\left(R_{\fp'}[x,f]\otimes_{R_{\fp'}}
E_{R_{\fp'}}(k(\fp'))\right)\right).
$$
However, by Examples \ref{jan.10}, \ref{jan.11} and \ref{jan.12},
Proposition \ref{jan.13}(iii) and Lemma \ref{nlrn.1},
\begin{align*}
\left(0:_{R_{\fp'}}\ann_{K}\left({\textstyle \bigoplus_{n\geq
h}R_{\fp'}(c''/1)^{p^j}x^n}\right)\right) & =
\left(0:_{R_{\fp'}}\ann_{E_{R_{\fp'}}(k(\fp'))}\left({\textstyle
\bigoplus_{n\geq h}R_{\fp'}(c''/1)^{p^j}x^n}\right)\right)\\ &=
\left(0:_{R_{\fp'}}\ann_{E_{R_{\fp'}}(k(\fp'))}\left({\textstyle
\bigoplus_{n\geq h}R_{\fp'}(c''/1)^{p^h}x^n}\right)\right)\\ &=
\left(0:_{R_{\fp'}}\left(0:_{E_{S_{\fq'}}(k(\fq))} \ft
S_{\fq'}\right)\right) = (\ft/\fa)R_{\fp'}
\end{align*}
because an $R_{\fp'}$-module and its Matlis dual have equal
annihilators (even though the local ring $R_{\fp'}$ need not be
complete). Therefore $$(\ft/\fa)R_{\fp'} \subseteq (0:_{R_{\fp'}}
\Delta_{R_{\fp'}}(R_{\fp'}[x,f]\otimes_{R_{\fp'}}E_{R_{\fp'}}(k(\fp')))).$$

Recall that $\height \ft/\fa > 0$; let $c'$ be an arbitrary element
of $(\ft/\fa)\cap R^{\circ} = (\ft/\fa) \setminus\{0\}$. Then, in
$R_{\fp'}$,
$$
c'/1 \in (0:_{R_{\fp'}}
\Delta_{R_{\fp'}}(R_{\fp'}[x,f]\otimes_{R_{\fp'}}
E_{R_{\fp'}}(k(\fp')))),$$ so that $c' \in (0:_R
\Delta_R(R[x,f]\otimes_R E_R(R/\fp')))$ by Proposition
\ref{feb.1}(iv). Note that this is true for all $\fp' \in\Spec(R)$.
In particular, $c' \in (0:_R \Delta_R(R[x,f]\otimes_R E_R(R/\fm)))$
for all $\fm \in \Max(R)$.

Now it is easy to see that there is a homogeneous isomorphism of
graded left $R[x,f]$-modules
$$
R[x,f] \otimes_R \left({\textstyle \bigoplus_{\fm \in
\Max(R)}E_R(R/\fm)} \right) \cong {\textstyle \bigoplus_{\fm \in
\Max(R)} \left(R[x,f]\otimes_RE_R(R/\fm)\right)},
$$
and to deduce from this that
$
c' \in \left(0:_R \Delta_R\left(R[x,f] \otimes_R \left({\textstyle
\bigoplus_{\fm \in \Max(R)}E_R(R/\fm)} \right)\right)\right).
$
Since $E := \bigoplus_{\fm \in \Max(R)}E_R(R/\fm)$ is an injective
cogenerator of $R$, Theorem \ref{pl.5c} shows that $c'$ is a big
test element for $R$. Now $(0:_R\Delta(R[x,f]\otimes_RE))$ is, by
Theorem \ref{pl.5c}, equal to the big test ideal of $R$. As $c'$ was
an arbitrary element of $(\ft/\fa)\cap R^{\circ} = (\ft/\fa)
\setminus\{0\}$ and such elements generate $\ft/\fa$, it follows
that $\ft/\fa$ is contained in the big test ideal of $R$.

We now return to consideration of the particular $\fq \in
\Var(\fa)$, chosen in the first paragraph of this proof, such that
$R_{\fq/\fa}$ is Gorenstein and weakly $F$-regular. Write $\fp :=
\fq/\fa$. Recall that our choice of $u$ is such that $ \left((\fa
S_{\fq})^{[p]} : \fa S_{\fq}\right) = (\fa S_{\fq})^{[p]} +
(u/1)S_{\fq}.$ Since $R_{\fp}$ is also excellent, Proposition
\ref{hrlr.12} shows that $E_{R_{\fp}}(k(\fp))$, with its
$R_{\fp}[x,f]$-module structure defined above (in the particular
case in which $\fp' = \fp$) is a simple left $R_{\fp}[x,f]$-module.
By our discussion above for a general $\fp' \in \Spec(R)$, we have
\begin{align*}
\left\{\xi \in E_{R_{\fp}}(k(\fp)) : (c''/1)^{p^h}x^n\xi= 0
\text{~for all~} n \geq h\right\} & = \left(0:_{E_{S_{\fq}}(k(\fq))}
\ft S_{\fq}\right)\\ & = \left(0:_{E_{R_{\fp}}(k(\fp))} (\ft/\fa)
R_{\fp}\right)\mbox{;}
\end{align*}
as this is an $R_{\fp}[x,f]$-submodule of $E_{R_{\fp}}(k(\fp))$, and
as $\height \ft/\fa >0$, we must have $(\ft/\fa) R_{\fp} = R_{\fp}$,
so that $\ft/\fa \not\subseteq \fp$. Thus the big test ideal of $R$
is not contained in $\fp$.

We have now shown that $(0:_R\Delta(R[x,f]\otimes_RE))$, the big
test ideal of $R$, is not contained in any  $\fp'' \in \Spec (R)$
for which $R_{\fp''}$ is Gorenstein and weakly $F$-regular. It
follows that, if $c \in R^{\circ}$ is such that $R_c$ is Gorenstein
and weakly $F$-regular, then $c \in
\sqrt{(0:_R\Delta(R[x,f]\otimes_RE))}$, so that some power of $c$ is
a big test element for $R$.
\end{proof}

\begin{cor} \label{jul.cdom} Suppose that $R$ is a homomorphic
image of a regular ring $S$ of characteristic $p$ that is
$F$-finite; assume that $R$ satisfies condition $(R_0)$. Let $c\in
R^{\circ}$ be such that $R_c$ is Gorenstein and weakly $F$-regular.
Then there is a power of $c$ that is a big test element for $R$.
\end{cor}

\begin{proof} Since $S$ is $F$-finite, it is excellent, by E. Kunz
\cite[Theorem 2.5]{Kunz76}; also, $S$ is $F$-$\cap$-flat, by an
earlier remark. The claim therefore follows from Theorem
\ref{nlrn.2}.
\end{proof}

We can enlarge the class of rings to which the ideas of this section
can be applied by use of Proposition \ref{jul.4}.

\begin{cor}\label{jul.5}
Suppose that $R \subseteq R'$ is a faithfully flat extension of
excellent rings of characteristic $p$ such that all the fibre rings
of the inclusion ring homomorphism are regular, and such that $R'$
is a homomorphic image of an excellent regular ring $S$ of
characteristic $p$ that is $F$-$\cap$-flat.

Suppose that $R$ satisfies condition $(R_0)$, and that $c \in
R^{\circ}$ is such that $R_c$ is Gorenstein and weakly $F$-regular.
Then some power of $c$ is a big test element for $R$.
\end{cor}

\begin{proof} It follows from Corollary \ref{sr.2} that $R'_c$ is
Gorenstein and weakly $F$-regular; note also that $c \in R'^{\circ}$
and that $R'$ satisfies $(R_0)$. By Theorem \ref{nlrn.2}, there
exists $n \in \N$ such that $c^n$ is a big test element for $R'$,
and we can conclude from Proposition \ref{jul.4} that $c^n$ is a big
test element for $R$.
\end{proof}

We can now give an example that has some resonance with the
Hochster--Huneke Theorem \ref{in.1}.

\begin{thm}\label{jul.8} Suppose that $R$ is an algebra of finite type over
an excellent local ring $A$ of characteristic $p$ whose residue
field $\K$ is $F$-finite. If $R$ satisfies the condition $(R_0)$,
then $R$ has a big test element.

In fact, if $c \in R^{\circ}$ is such that $R_c$ is Gorenstein and
weakly $F$-regular, then some power of $c$ is a big test element for
$R$.
\end{thm}

\begin{proof} Let $\phi :A \lra \widehat{A}$ be the inclusion
homomorphism. Since $A$ is excellent, $\phi$ is a regular
homomorphism (see \cite[p.\ 256]{HM}). By \cite[p.\ 253]{HMold} the
induced (faithfully flat) ring homomorphism $R \lra \widehat{A}
\otimes_AR$ is also regular.

By I. S. Cohen's Structure Theorem for complete local rings,
$\widehat{A}$ is a homomorphic image of $\K[[Y_1, \ldots, Y_m]]$ for
some indeterminates $Y_1, \ldots, Y_m$. Therefore $\widehat{A}
\otimes_AR$ is a homomorphic image of $\K[[Y_1, \ldots, Y_m]][X_1,
\ldots, X_n]$, and this is an excellent $F$-$\cap$-flat regular
ring, by Example \ref{if.2}(iii). The claim therefore follows from
Corollary \ref{jul.5}.
\end{proof}

In one respect, Theorem \ref{jul.8} is inferior to the
Hochster--Huneke Theorem \ref{in.1} because it only applies to an
algebra of finite type over an excellent local ring of
characteristic $p$ that has $F$-finite residue field; on the other
hand, it establishes the existence of big test elements and applies
when $R$ satisfies the condition $(R_0)$ (as well as the other
hypotheses).

Finally, it should be noted that arguments have been provided in
this paper so that its results can be proved without use of the
Gamma construction.

\end{document}